\documentclass[article,11pt]{amsart}
\usepackage[top=25mm,bottom=25mm,left=25mm,right=25mm]{geometry}

\newtheorem{lemma}{Lemma}[section]
\newtheorem{proposition}{Proposition}[section]
\newtheorem{theorem}{Theorem}[section]
\newtheorem{corollary}{Corollary}[section]
\newtheorem{definition}{Definition}[section]
\newtheorem{example}{Example}[section]
\newtheorem{remark}{Remark}[section]

\usepackage{color}

\makeatletter
\def\section{\@startsection{section}{1}%
\z@{1\linespacing\@plus\linespacing}{1\linespacing}%
{\bf\centering}}
\def\subsection{\@startsection{subsection}{0}%
\z@{\linespacing\@plus\linespacing}{\linespacing}%
{\bf}}
\makeatother

\makeatletter
\@addtoreset{equation}{section}
\makeatother

\DeclareMathOperator{\Id}{{\rm Id}}
\DeclareMathOperator{\supp}{supp}

\DeclareMathOperator{\Spec}{Spec}

\DeclareMathOperator{\dist}{dist}

\DeclareMathOperator{\loc}{loc}

\providecommand{\pro}[1]{(#1_t)_{t \geq 0}}

\providecommand{\seq}[1]{(#1_n)_{n\in \mathbb{N}}}

\newcommand{\cK}{\mathcal{K}}

\newcommand{\cB}{\mathcal{B}}

\newcommand{\cE}{\mathcal{E}}

\newcommand{\R}{\mathbf{R}}
\newcommand{\1}{\mathbf{1}}

\newcommand{\pr}{\mathbf{P}}

\newcommand{\ex}{\mathbf{E}}

\newcommand{\Rd}{\mathbf{R}^d}
\newcommand{\N}{\mathbf{N}}

\begin{document}
\title[Fall-off for non-local Schr\"odinger operators with decaying potentials]
{Fall-off of eigenfunctions for non-local Schr\"odinger operators with decaying potentials}
\author{Kamil Kaleta and J\'ozsef L\H orinczi}
\address{Kamil Kaleta, Institute of Mathematics \\ University of Warsaw \\
ul. Banacha 2, 02-097 Warszawa and Faculty of Pure and Applied Mathematics \\
 Wroc{\l}aw University of Science and Technology
\\ Wyb. Wyspia{\'n}skiego 27, 50-370 Wroc{\l}aw, Poland}
\email{kamil.kaleta@pwr.edu.pl, kkaleta@mimuw.edu.pl}

\address{J\'ozsef L\H orinczi,
Department of Mathematical Sciences, Loughborough University \\
Loughborough LE11 3TU, United Kingdom}
\email{J.Lorinczi@lboro.ac.uk}

\thanks{
\noindent
2010 {\it MS Classification}: Primary 47D08, 60G51; Secondary 47D03, 47G20 \\
\noindent
K. Kaleta was supported by the National Science Center (Poland) grant 2012/04/S/ST1/00093 and by the Foundation for Polish Science.
We both are grateful to IH\'ES, Bures-sur-Yvette, for a research stay where this work has started.
}

\begin{abstract}
We study the spatial decay of eigenfunctions of non-local Schr\"odinger operators whose kinetic terms are generators of symmetric
jump-paring L\'evy processes with Kato-class potentials decaying at infinity. This class of processes has the property
that the intensity of single large jumps dominates the intensity of all multiple large jumps. We find that the decay rates
of eigenfunctions depend on the process via specific preference rates in particular jump scenarios, and depend on the
potential through the distance of the corresponding eigenvalue from the edge of the continuous spectrum. We prove that
the conditions of the jump-paring class imply that for all eigenvalues the corresponding positive eigenfunctions decay
at most as rapidly as the L\'evy intensity.
This condition is sharp in the sense that if the jump-paring property fails to hold, then eigenfunction decay becomes
slower than the decay of the L\'evy intensity. We furthermore prove that under reasonable conditions the L\'evy intensity
also governs the upper bounds of eigenfunctions, and ground states are comparable with it, i.e., two-sided bounds hold.
As an interesting consequence, we identify a sharp regime change in the decay of eigenfunctions as the L\'evy intensity is varied
from sub-exponential to exponential order, and dependent on the location of the eigenvalue, in the sense that through the
transition L\'evy intensity-driven decay becomes slower than the rate of decay of the L\'evy intensity.
Our approach is based on path integration and probabilistic potential theory techniques,
and all results are also illustrated by specific examples.

\end{abstract}

\maketitle

\baselineskip 0.5 cm

\bigskip
\section{Introduction}
\noindent
Non-local Schr\"odinger operators of the form
$$
H = H_0 + V,
$$
where $H_0$ is a non-local (pseudo-differential, see \cite{bib:J}) operator giving the kinetic term and $V$ is a multiplication
operator called potential, receive increasing attention in both pure and applied mathematics. They pose intriguing
problems in the intersection areas of functional analysis, PDE and probability, and offer a new framework in
scientific modelling, providing life-like correctives and refinements to established theories. In the present
paper we are focusing on some spectral properties of such operators by developing a stochastic approach via
Feynman-Kac type representations.

There are few examples in which the spectrum and the eigenfunctions of a non-local Schr\"odinger operator are
explicitly known, see \cite{bib:LMa,bib:DL14}. When such detailed expressions are not available, it is
natural to ask about the spatial decay properties of eigenfunctions in function of $V$ at least in terms of estimates.
A basic interest in this property is that it tells of how well a quantum particle described by $H$ is localized in
physical space or, alternatively, of the concentration of mass in the stationary measure of the random process
corresponding to $H$.

The pointwise decay of eigenfunctions of Schr\"odinger operators, i.e., when $H_0 = -\frac{1}{2}\Delta$ and the
underlying random process is Brownian motion, has been much studied in the literature, see e.g. \cite{bib:A,bib:Ca,bib:S82,bib:S}
and the references therein. Suppose $\lambda \in \Spec H$ is an isolated eigenvalue and $\varphi \in L^2(\R^d)$ is a corresponding
eigenfunction of $H$, which is called a ground state and will be denoted by $\varphi_0$ when the eigenvalue lies at the bottom of
the spectrum, i.e., $\lambda_0 = \inf \Spec H$. When $V$ is a confining potential in the sense that $V(x) \to \infty$ as
$|x| \to \infty$, the spectrum of $H$ is purely discrete and a typical answer is that the decay of eigenfunctions is exponential
or faster. For instance, if $V(x) \asymp |x|^{2\beta}$, $\beta   > 1$, then the ground state decays super-exponentially like
$$
\varphi_0(x) \asymp |x|^{-\frac{1+\beta-d}{2}}e^{-\frac{1}{1+\beta}|x|^{1 + \beta}},
$$
i.e., the rate of decay is directly determined by the rate of asymptotic growth of the potential. In this case the semigroup
$\{e^{-tH}: t\geq 0\}$ is intrinsically ultracontractive, and the other eigenfunctions are asymptotically dominated by
the ground state \cite{bib:DS,bib:B}. When the potential is decaying, i.e., $V(x) \to 0$ as $|x| \to \infty$,
the situation is more delicate.
In such cases the discrete component of the spectrum may be empty and thus eigenfunctions may not exist at all.
However, if an eigenfunction does exist, then under some further conditions on the potential it can be shown that it still
decays exponentially \cite{bib:A,bib:Ca}. A typical result is then that whenever $V \in L^\infty_{\rm loc}(\Rd)$ such that
$\liminf_{|x|\to\infty}V(x) \geq 0$ and $\varphi$ is an eigenfunction for eigenvalue $ \lambda < 0$, then for every
$\varepsilon \in (0, |\lambda|)$ there exists $C_\varepsilon   > 0$ such that
\begin{equation}
\label{decayinV}
|\varphi(x)| \leq C_\varepsilon e^{-\sqrt{\frac{|\lambda|-\varepsilon}{2}}|x|}.
\end{equation}

When in $H$ the term $H_0$ is replaced by a non-local operator, the behaviours change essentially. In the paper \cite{bib:KL14}
we have considered not just one but a whole class of operators $H_0$ corresponding to the generators of symmetric L\'evy
processes with the property that all multiple large jumps are dominated under the L\'evy measure with density $\nu$ by a single
large jump (which we call jump-paring L\'evy processes, see Section \ref{sec:Prel} below). This class includes important
examples such as $H_0 = (-\Delta + m^{2/\alpha})^{\alpha/2} -m$, $0 <\alpha <2$, with polynomially decaying L\'evy measure for
$m=0$ (isotropic $\alpha$-stable processes, see \cite{bib:KL, bib:KaKu}), and exponentially localized L\'evy measure for $m > 0$
(relativistic $\alpha$-stable processes), and has a non-trivial overlap with the class of non-local operators obtained under
Bernstein functions of the Laplacian \cite{bib:HIL12}. Taking a sufficiently regular Kato-class confining potential $V$, we showed
that the ground state of $H$ behaves like
\begin{align}
\label{decayinnuV}
\varphi_0(x) \asymp \frac{\nu(x)}{V(x)}.
\end{align}
This gives a neat account of the separate contributions of the unperturbed process and of the perturbation into the decay.
For general Kato-class potentials, \eqref{decayinnuV} has the generic form
\begin{align}
\label{decayinnuextime}
\varphi_0(x) \asymp \nu(x) \, \Lambda_{V}(x),
\end{align}
where $\Lambda_{V}(x)$ is the mean exit time from a unit ball centered at the starting point $x$ of the process under
the potential $V$. This formula gives a probabilistic interpretation to the decay of ground states, from which the above
estimate is also derived. Moreover, we proved that the other eigenfunctions $\varphi_n$, $n \in \Bbb N$, satisfy
$$
|\varphi_n(x)| \leq C_n  \varphi_0(x), \quad x \in \R^d,
$$
with a suitable constant $C_n(X,V)>0$, dependent on the process and the potential. We emphasize that this ground state
domination follows in the non-local case with no involvement and also in lack of intrinsic ultracontractivity. For the
fractional Laplacian ($m=0$ above) and the same potential $V(x) \asymp |x|^{2\beta}$ as above, this means now a much
slower decay like
$$
\varphi_0(x) \asymp \frac{1}{|x|^{\alpha + d + 2\beta}}.
$$
The picture for confining potentials corresponds to the intuition that in this case there is a killing mechanism with
increasing values of the potential, and the tails of the ground state should depend on the balance between the potential
and the L\'evy intensity. In terms of techniques, these results have been established by using sharp uniform estimates
on the local extrema of harmonic functions for the perturbed L\'evy process combined with a close systematic control of
jumps, based on the jump-paring property (see Definition \ref{jumpparing} below).

Eigenfunction decay for non-local Schr\"odinger operators with potentials decreasing to zero has been little understood so
far. In the seminal paper \cite{bib:CMS90} it was considered for the operators $H_0=(-\Delta)^{\alpha/2}$, $0 < \alpha < 2$,
and $H_0 = \sqrt{-\Delta + m^2}-m$, $m   > 0$, using martingale and optional stopping methods combined with precise estimates
of the corresponding resolvent kernels. For an extension of these ideas to other operators of interest in mathematical physics
see \cite{bib:HIL13, bib:LHB}.

In the present work we use the wide framework of symmetric jump-paring processes introduced in \cite{bib:KL14}, which has the 
advantage of accommodating a large selection of interesting types of L\'evy processes without being too abstract, and focus here 
on the complementary and essentially different case of potentials decreasing to zero at infinity. Our aim is to derive the 
fall-off behaviour of eigenfunctions in function of the process and the potential. In particular, a main question we address is if
the basic relationship \eqref{decayinnuextime} continues to hold for decaying potentials. For a potential $V$ decaying to zero
the perturbed processes behave far out like free processes, thus we have $\Lambda_{V}(x) \asymp \mbox{const}$ and it can be
expected that $\varphi_0 \asymp \nu$. This means that there is an essential difference from the confining case in that there
is no longer a balancing mechanism as now both $\nu$ and $V$ decrease with the distance from the origin. A main consequence
is that the contribution of both the process and the potential in the fall-off rates of eigenfunctions is now more subtle.
The effect of the potential appears in the relative position of the corresponding eigenvalue from the edge of the continuous spectrum.
The effect of the process comes in mainly through two parameter functions expressing (inverse) preference rates for specific
jump scenarios. They are defined in \eqref{def:K1}-\eqref{def:K2} below and discussed in detail; a third parameter function
given by \eqref{def:K_3} plays a technical role only.

Our main results can be roughly summarized as follows.
\begin{enumerate}
\item[(1)]
Whenever $\varphi$ is a positive eigenfunction corresponding to an eigenvalue $\lambda \in \R$, it is bounded from below by the
L\'evy intensity (Theorem \ref{prop:lowerbound}), i.e.,
$$
\varphi(x) \geq \mbox{const} \, \nu(x)
$$
for large enough $|x|$, with a prefactor dependent on the process and the eigenvalue. Moreover, if the basic jump-paring
condition (A1.3) below does not hold, then $\varphi$ necessarily decays slower than $\nu$.

\item[(2)]
When an eigenvalue $\lambda$ is sufficiently low-lying below zero, the corresponding
eigenfunction $\varphi$ satisfies
$$
|\varphi(x)| \leq \mbox{const} \left\|\varphi\right\|_{\infty} \nu(x),
$$
for large enough $|x|$, where the constant prefactor depends on the process and the eigenvalue (Theorem \ref{thm:defic6}).
When $\lambda < 0$ is arbitrary, we consider separately symmetric jump-paring processes with L\'evy densities that are
slowly (see \eqref{eq:nu_dbl}) or fast (see \eqref{eq:fsm}-\eqref{eq:comp_dens}) decaying at infinity. Under a smallness condition
involving the parameter functions, see \eqref{eq:intr_killing}-\eqref{eq:unif_bdd} and \eqref{eq:Green_bdd}, we find that the upper
bound of $\varphi$ is again driven by the jump intensity as above (Theorem \ref{thm:defic4}).

\item[(3)]
A combination of the upper and lower bounds in (1)-(2) above gives that when an eigenfunction is a ground state, we have
$$
\varphi_0(x) \asymp \nu(x).
$$
This then holds for low-lying bottom eigenvalues $\lambda_0$, and any negative $\lambda$ whenever the conditions in (1) are in place.
Also, from the above it is seen that \eqref{decayinnuextime} does apply to the case of decaying potentials in these circumstances. The
smallness of the ground state eigenvalue can be understood in terms of the cost of passing a potential barrier plus leaving a ball
(Proposition \ref{smallev}).

\item[(4)]
Our present framework offers a unified treatment of both decaying and confining potentials. In Theorem \ref{thm:defic_conf} we are
able to derive an upper bound for eigenfunctions in the case of confining potentials by using the methods developed in this paper,
and recover a result in \cite{bib:KL14} as a bonus.
\end{enumerate}

It is helpful to see to what expressions these results translate in some specific cases. From these behaviours the following
interesting phenomenon emerges. For a jump-paring process with L\'evy intensity $\nu$ decaying slower than exponentially (e.g.,
polynomially or sub-exponentially heavy-tailed) the corresponding ground state has the same fall-off rate as $\nu$ (see Corollaries
\ref{thm:polynomial}-\ref{thm:subexp} and Remarks \ref{rem:poly}-\ref{rem:inter}). When $\nu$ decays exponentially or faster, the
regime qualitatively changes and a ``phase transition" in the fall-off rates can be observed. For exponentially decaying L\'evy
intensity the following dichotomy occurs. If the ground state is an eigenfunction at a sufficiently low-lying eigenvalue, it has the
same fall-off as $\nu$, while for bottom eigenvalues which are closer to zero (i.e., to the edge of the continuous spectrum), the fall-off
gets much slower, with essential contribution of the eigenvalue into the rate, see Corollary \ref{prop:exp}. (Suggestively, this
means that the ground state of such a process in a deep enough potential well decays like $\nu$, decays slower than $\nu$ if the
well is not deep enough, and the ground state may even cease to exist if the well is too shallow.) When the process is outside of the
jump-paring class, the ground state decays slower than $\nu$ (Theorem \ref{prop:lowerbound}), with significant contribution from the
bottom eigenvalue as long as its absolute value is not too large (see Corollaries \ref{prop:exp}-\ref{prop:supexp}). In particular, the
resulting fall-off becomes comparable to that of perturbed Brownian motion \eqref{decayinV}. We note that this phenomenon can be appreciated as another level in the hierarchy of ground state behaviour, and the order of  increasingly ``regular" properties can be seen as the line evolving from slower than L\'evy intensity-driven decay, through L\'evy intensity-driven decay or faster, and intrinsic ultracontractivity which is topping this by a uniform ergodic behaviour. For further details see Subsection \ref{subsec:phase} and Remark \ref{rem:exp_trans}, and for comparison with confining potentials we refer to \cite[Sects. 2.3-2.5, 4.2]{bib:KL14}.

The mechanism behind these differing behaviours can be heuristically understood as follows. For illustration consider a negative
potential well with a single minimum at the origin, tending to zero as $|x|\to\infty$, with ground state eigenvalue $\lambda_0<0$.
For energetic reasons the paths will be  encouraged to move and spend long times in close neighbourhoods of the bottom of the well
(a concentration effect), while for entropic reasons the process tends to explore space arbitrarily far. Furthermore, a lower ground
state eigenvalue means a more prominent energetic effect with paths concentrating around the bottom of the well, resembling the case of
confining potentials. Stronger concentration means that paths are better localized in space and so the ground state has a faster decay.
In this case the fall-off rate is the best possible, and for the
processes considered it is given by $\nu$. On the other hand, when $|\lambda_0|$ is not large enough, the energetic effect weakens and
the efficiency of the concentration mechanism loses out so that the perturbed process behaves more like the free process. However, for
a process starting far away from the origin and moving via long direct jumps (which is the case for a jump-paring process) rather than
via sequences of long jumps interspersed with
smaller fluctuations, this relatively small energetic effect can still be enough for securing an efficient concentration. This again
gives a better localization of paths and implies a decay of the ground state determined by $\nu$. The jump property described above is
formally expressed by conditions \eqref{eq:intr_killing}-\eqref{eq:unif_bdd} and we think of it as the capacity of responsiveness to
perturbation of the given process. To complete this picture, we note that for jump-paring processes without this property or for other
processes moving in smaller jumps or fluctuating continuously, the paths from far out can cluster around the potential well less
efficiently and once they are back around the origin, they spend comparatively large amounts of time building up ``backlogs" in the
decay-events, so their ground states decay much slower than $\nu$.

In what follows we develop a new methodology to tackle perturbations of jump processes accommodating subtle spectral effects coming
from potentials decaying to zero rather than producing a large killing by growing to infinity at infinity. Using a probabilistic
representation, the spatial decay of eigenfunctions becomes equivalent to the behaviour of mean hitting times of large balls centered at the
origin. The Laplace transforms of such hitting times are harmonic functions with respect to the process killed at a rate given by the absolute
value of the corresponding eigenvalue.
A first technical novelty in this paper is then a full deployment of a sequence of self-improving estimates on functions which we call harmonic
at infinity for the underlying jump-paring L\'evy process killed at a non-zero rate (Lemma \ref{lem:defic1}). We find that the resulting
upper bound driven by $\nu$ holds under the general balancing condition \eqref{eq:cond1} involving the killing rate and the pivotal parameter
functions \eqref{def:K1}-\eqref{def:K_3}. In Theorem \ref{thm:defic2} we show that this condition is satisfied for L\'evy densities with
a doubling property for large arguments. Theorem \ref{thm:defic3}, which applies to jump-paring processes with light jump intensities, is one
of the most involved and crucial technical results in the paper. It states that under the key smallness conditions on the parameter functions
\eqref{eq:intr_killing}-\eqref{eq:unif_bdd}, the Laplace transform of the first hitting time of a closed ball with respect to the path
measure of the process starting far from the origin is dominated by $\nu$. In deriving this result we come to a delicate argument based on the
domination of $\nu$ by a carefully constructed family of jump intensities to which the fundamental Lemma \ref{lem:defic1} can be applied.
We stress that improved estimates in Lemma \ref{lem:defic1} also allow to treat confining potentials now within the same framework, and we
are able to recover the decay estimates obtained in \cite{bib:KL14}, see Theorem \ref{thm:defic_conf} below. In particular, this shows that
if one is interested only in decay properties (and not also in intrinsic ultracontractivity), a detailed tracking of the jumps as done in
the cited paper is not needed, and we obtain a unified and streamlined treatment. We note that all upper estimates
obtained in the present paper are sharp in the sense that they are directly governed by $\nu(x)$ rather than $\nu(cx)$ with some $c \in
(0,1)$, even when $\nu$ is very light at infinity.

The remainder of this paper is organized as follows. In Section 2 we introduce and discuss the class of underlying L\'evy processes, the
corresponding parameter functions, potentials and Feynman-Kac semigroups considered in this paper. Section 3 is devoted to proving the
technical estimates for functions harmonic at infinity. In Section 4 we come to present the fall-off properties of eigenfunctions,
including a detailed discussion of decay rates for specific classes of processes and operators.


\bigskip
\section{Jump-paring L\'evy processes and non-local Schr\"odinger operators}

\subsection{Jump paring class of L\'evy processes}

\label{sec:Prel}
In this subsection we introduce the class of L\'evy processes considered in this paper and discuss some of their
properties which will be used below.

Let $\pro X$ be a symmetric L\'evy process with values in $\R^d$, $d \geq 1$, with probability measure $\pr^x$
of the process starting from $x \in \R^d$. We use the notation $\ex^x$ for expectation with respect to $\pr^x$.
Recall that $\pro X$ is a Markov process with respect to its natural filtration, satisfying the strong Markov
property and having c\`adl\`ag paths. It is determined by the characteristic function
$$
\ex^0 \left[e^{i \xi \cdot X_t}\right] = e^{-t \psi(\xi)}, \quad \xi \in \R^d, \ t >0,
$$
with the characteristic exponent given by the L\'evy-Khinchin formula
\begin{align} \label{eq:Lchexp}
\psi(\xi) = A \xi \cdot \xi + \int_{\R^d} (1-\cos(\xi \cdot z)) \nu(dz).
\end{align}
Here $A=(a_{ij})_{1\leq i,j \leq d}$ is a symmetric non-negative definite matrix, and $\nu$ is a symmetric
L\'evy measure on $\R^d \backslash \left\{0\right\}$, i.e., $\int_{\R^d} (1 \wedge |z|^2) \nu(dz) < \infty$ and
$\nu(E)= \nu(-E)$, for every Borel set $E \subset \R^d \backslash \left\{0\right\}$, thus the L\'evy triplet of the
process is $(0,A,\nu)$.

In the present paper we assume throughout that the L\'evy measure appearing in \eqref{eq:Lchexp} is an infinite
measure and it is absolutely continuous with respect to Lebesgue measure, i.e.,
\begin{align} \label{eq:nuinf}
\nu(\R^d \backslash \left\{0\right\})=\infty \quad \text{and} \quad \nu(dx)=\nu(x)dx, \quad \text{with}
\quad \nu(x)   > 0.
\end{align}
For simplicity, we denote the density of the L\'evy measure also by $\nu$ as it is the object we will use below.
When $A \equiv 0$, the random process $\pro X$ is said to be a
purely jump process. Note that the properties (\ref{eq:nuinf}) jointly imply that $\pro X$ is a strong Feller
process, or equivalently, its one-dimensional distributions are absolutely continuous with respect to Lebesgue
measure, i.e., there exist measurable transition densities $p(t,x,y) = p(t,0,y-x)=: p(t,y-x)$ such that $\pr^0(X_t \in E) =
\int_E p(t,x)dx$, for every Borel set $E \subset \R^d$ (see e.g. \cite[Th. 27.7]{bib:Sat}).

We will make use below of the following symmetrization of the exponent $\psi$. Denote
\begin{align} \label{eq:Lchexpprof}
\Psi(r) = \sup_{|\xi| \leq r} \psi(\xi), \quad r>0.
\end{align}
It follows from a combination of \cite[Rem. 4.8]{bib:Sch} and \cite[Sect. 3]{bib:Pru} (see also direct
calculations with explicit constants in \cite[Lem. 4]{bib:Grz}) that there exist $C_1, C_2 > 0$, independent of
the process (i.e., of $A$ and $\nu$), such that
\begin{align} \label{eq:PruitH}
C_1 H\left(\frac{1}{r}\right) \leq \Psi(r) \leq C_2 H\left(\frac{1}{r}\right), \;\; r>0, \quad \text{where}
\quad H(r) = \frac{\left\|A\right\|}{r^2}+ \int_{\Rd} \left(1 \wedge \frac{|y|^2}{r^2}\right) \nu(dy).
\end{align}
It can be directly checked that $H$ is non-increasing and the doubling property $H(r) \leq 4 H(2r)$, $r>0$, holds.
In particular, it follows that $\Psi(2r) \leq 4 C_1^{-1} C_2 \Psi(r)$, for all $r>0$.

\medskip

The generator $L$ of the process $\pro X$ is uniquely determined by its Fourier symbol
\begin{align} \label{def:gen}
\widehat{L f}(\xi) = - \psi(\xi) \widehat{f}(\xi), \quad \xi \in \R^d, \; f \in D(L),
\end{align}
with domain $D(L)=\left\{f \in L^2(\R^d): \psi \widehat f \in L^2(\R^d) \right\}$. It is
a negative non-local self-adjoint operator with core $C_0^{\infty}(\R^d)$, and
$$
L f(x) = \sum_{i,j=1}^d a_{ij} \frac{\partial^2 f}{\partial x_j \partial x_i} (x)
+  \lim_{\varepsilon \searrow 0} \int_{|y-x|>\varepsilon} (f(y)-f(x))\nu(y-x)dy, \quad x \in \R^d,
\; f \in C_0^{\infty}(\R^d).
$$
The corresponding Dirichlet form $(\cE, D(\cE))$ is defined by
\begin{align} \label{def:qform}
\cE(f,g) = \int_{\R^d} \psi(\xi) \widehat{f}(\xi) \overline{\widehat{g}(\xi)} d\xi, \quad f, g \in D(\cE),
\end{align}
where $D(\cE)=\left\{f \in L^2(\R^d): \int_{\R^d} \psi(\xi) |\widehat{f} (\xi)|^2 d\xi < \infty \right\}$.
If $f \in D(L)$, then $\cE(f,g)=(-Lf,g)$. Furthermore, let
$$
\cB = \left\{f \in C^2_{\rm c}(\R^d): f(x)=1 \ \text{for} \ x \in B(0,1/2),
f(x)=0 \ \text{for} \ x \in B(0,1)^c \ \text{and} \ 0 \leq f \leq 1\right\},
$$
and notice that for $f_s(x) = f(x/s)$ with $f \in \cB$ and $s   >0$, we have
\begin{eqnarray*}
\left\|L f_s \right\|_{\infty}
& \leq &
2 \nu(B(0,s)^c) + \frac{1}{s^2}\sup_{i,j=1,...,d} \left\|\frac{\partial^2 f}{\partial x_i \partial x_j}\right\|_{\infty}
\left(\left\|A\right\| + \int_{|y|  \leq s}|y|^2 \nu(dy)\right) \\
& \leq &
\left(2 \vee \sup_{i,j=1,...,d} \left\|\frac{\partial^2 f}{\partial x_i \partial x_j}\right\|_{\infty} \right) H(s),
\quad s >0. \nonumber
\end{eqnarray*}
Denote
\begin{eqnarray}\label{def:C9a}
C_3(X, s) := \inf_{f \in \cB} \left\|L f_s \right\|_{\infty}, \quad \text{with} \quad f_s(x) = f(x/s), \ \ \ \ s>0.
\end{eqnarray}
By the above we have
\begin{align}\label{def:C9}
C_3(X,s) \leq C^{-1}_1 \left(2 \vee \inf_{f \in \cB} \sup_{i,j=1,...,d} \left\|\frac{\partial^2 f}{\partial x_i
\partial x_j}\right\|_{\infty}  \right) \Psi(1/s), \quad s >0.
\end{align}
For more details
on L\'evy processes we refer to \cite{bib:Ber,bib:App,bib:Sat,bib:J}.

Let $D \subset \R^d$ be an open bounded set and consider the first exit time $\tau_D = \inf\left\{t \geq 0: X_t \notin D\right\}$
from $D$. The transition densities $p_D(t,x,y)$ of the process killed upon exiting $D$ are given by the Dynkin-Hunt formula
\begin{align}
\label{eq:HuntF}
p_D(t,x,y)= p(t,y-x) - \ex^x\left[ \tau_D < t; p(t-\tau_D, y-X_{\tau_D})\right], \quad x , y \in D.
\end{align}
The Green function of the process $\pro X$ on $D$ is thus $G_D(x,y)= \int_0^{\infty} p_D(t,x,y) dt$, for all $x, y \in D$, 
and $G_D(x,y)=0$ if $x \notin D$ or $y \notin D$.
Furthermore, by \cite[Rem. 4.8]{bib:Sch} we have
\begin{align} \label{eq:gen_est_tau}
\ex^0[\tau_{B(0,r)}] \leq  \frac{C_{4}}{\Psi(1/r)}, \quad r    > 0,
\end{align}
with a constant $C_{4}$ independent of the process.

We will use throughout the notation $C(a,b,c,...)$ for a positive constant dependent on parameters $a,b,c,...$,
while dependence on the process $X:=\pro X$ is indicated by $C(X)$, and dependence on the dimension $d$ is assumed
without being stated explicitly. Since constants appearing in definitions, lemmas and theorems play an important
role in this paper, we use the numbering $C_1, C_2, ...$ to be able to track them. We will also use the notation
$f \asymp C g$ meaning that $C^{-1} g \leq f \leq Cg$ with a constant $C \geq 1$, while $f \asymp g$ means that there is
a constant $C \geq 1$ such that the latter holds. By $f \approx g$ we understand that $\lim_{|x| \to \infty} f(x)/g(x)=1$.
In proofs $c_1, c_2, ...$ will be used to denote auxiliary constants.

We will use the following class of L\'evy processes.
\medskip
\begin{definition}[\textbf{Symmetric jump-paring L\'evy processes}]
\label{jumpparing}
Let $\pro X$ be a L\'evy process with L\'evy-Khinchin exponent $\psi$ as in \eqref{eq:Lchexp}--\eqref{eq:nuinf} and
L\'evy triplet $(0,A,\nu)$, satisfying the following conditions.
\begin{itemize}
\item[\textbf{(A1)}] \texttt{L\'evy intensity:}
There exist a profile function $g:(0,\infty) \to (0,\infty)$ and a constant $C_5=C_5(X)$ such that
$$
\nu(x) \asymp C_5 g(|x|), \quad x \in \R^d \backslash \left\{0\right\},
$$
and the following properties hold:
\begin{itemize}
\item[(A1.1)]
$g$ is non-increasing on $(0,\infty)$
\item[(A1.2)]
there exists a constant $C_6=C_6(X) $ such that
\begin{align*}
g(|x|) \leq C_6 g(|x|+1), \quad |x| \geq 1
\end{align*}
\item[(A1.3)]
there exists a constant $C_7=C_7(X)$ such that
\begin{align*}
\int_{|x-y|>1 \atop |y| > 1} g(|x-y|) g(|y|) dy \leq C_7 \: g(|x|), \quad |x|\geq 1.
\end{align*}
\end{itemize}
\medskip
\item[\textbf{(A2)}] \texttt{Transition density:}
There exists $t_{\rm b} >0$ such that $\sup_{x\in \R^d} p(t_{\rm b},x) = p(t_{\rm b},0) < \infty$.
\medskip
\item[\textbf{(A3)}] \texttt{Green function:}
For all $0<p<q<r < \infty$ we have
$$
\sup_{x \in B(0,p)} \sup_{y \in B(0,q)^c} G_{B(0,r)}(x,y) < \infty.
$$
\end{itemize}
We call $\pro X$ satisfying the above conditions a \emph{symmetric jump-paring L\'evy process} and refer to the convolution condition in (A1.3) as the \emph{jump-paring property}.
\end{definition}
\noindent
Assumptions (A1.1)-(A1.2) are self-explanatory. It can be directly shown that these conditions and a similar geometric
argument as in \cite[Lem. 3.4(2)]{bib:KL14}) imply
\begin{align} \label{eq:b2}
\int_{r < |z| \leq r+1 \atop |y-z|   >1/8} \nu(z-y)dz \leq C_8 \int_{r-1 < |z| \leq r} \nu(z-y)dz, \quad
|y| \geq r+1, \; r \geq 1,
\end{align}
with a constant $C_8=C_8(X) \geq 1$, independent of $r$. The bound in (A1.3) provides a control of the convolutions
of $\nu$ with respect to large jumps and has a structural importance in defining the class of processes we consider.
It says that the intensity of double large jumps of the process are dominated by the intensity of a single large jump.
Let $\nu_{1}(x) = \nu(x) \1_{B(0,1)^c}(x)$. It is then seen iteratively that under (A1.3) in fact
$$
\nu_{1}^{n*}(x) \leq C^{n-1} \nu_{1}(x), \quad |x| \geq 1, \; n \in \N,
$$
holds, which means that every sequence of any finite length of large jumps of the process is dominated by single large
jumps, which gives the name to the class of L\'evy processes above.

The convolution condition (A1.3) has been introduced in \cite{bib:KL14} and proved to be a strong tool in studying large-scale
properties of jump L\'evy processes. Recently, in \cite{bib:KS14} it was also used to characterize the short-time behaviour
of heat kernels for a large class of convolution semigroups. It can be easily checked that (A1.3) in fact implies (A1.2), see
\cite[Lem. 1(a)]{bib:KS14}, however, for completeness and more clarity we prefer to state (A1.2) in Definition \ref{jumpparing}
separately.

Assumption (A2) is equivalent
with $e^{-t_{\rm b} \psi} \in L^1(\R^d)$, for some $t_{\rm b}   >0$. In this case $p(t_{\rm b},x)$ can be obtained
by the Fourier inversion formula. Clearly, this property extends to all $t \geq t_{\rm b}$ by the Markov property
of $\pro X$. For more details on the existence and properties of transition probability densities for L\'evy
processes we refer to \cite{bib:KSch} and references therein. Finally, we remark that in many cases of interest
Assumption (A3) follows directly from (rough) space-time estimates of the densities $p(t,x)$. Indeed, if Assumption (A2)
holds and for every $r>0$ there exists $C=C(r)$ such that $\sup_{|x|\geq r} p(t,x) \leq C t$, $t>0$, then (A3) follows
by a standard estimate as in \eqref{eq:est3} below.

The range of processes satisfying Assumptions (A1)-(A3) is wide including large subclasses of isotropic unimodal
L\'evy processes, subordinate Brownian motions, L\'evy processes with non-degenerate Brownian components, symmetric
stable-like processes, or processes with subexponentially or exponentially localized L\'evy measures. In particular,
it covers all the examples discussed in detail in \cite[Sect. 4]{bib:KL14}.

\subsection{Parameter functions and jump scenarios}
Next we introduce some functions of the L\'evy processes considered through which the decay of eigenfunctions of the
non-local Schr\"odinger operators will be analyzed. As it will be seen, they relate to some special features of the
process, which to our knowledge have not been addressed in the literature before.

\begin{enumerate}
\item
Define
\begin{equation}
\label{def:K1}
K^X_1(s):= \sup_{|x| \geq s} \frac{\int_{|x-y|>s, \, |y| > s} \nu(x-y) \nu(y) dy}{\nu(x)}, \quad s \geq 1.
\end{equation}
Notice that from (A1.3) it follows that $K^X_1:[1, \infty) \to (0,C_5^3C_7]$ is a non-increasing function and gives
the optimal constant $C$ in the bound
$$
\int_{|x-y|>s \atop |y| > s} \nu(x-y) \nu(y) dy \leq C \nu(x), \quad |x| \geq s,
$$
for any fixed $s \geq 1$. Thus $1/K^X_1(s)$ is the rate of preference of single jumps of size at least $s$
over double jumps of size at least $s$ each, i.e., when $K^X_1(s)$ decreases with $s \to \infty$, this preference improves.

\medskip
\item
Let $0<s_1 < s_2 < s_3 \leq \infty$ and define
\begin{align} \label{def:K2}
K^X_2(s_1,s_2,s_3):= \inf\left\{ C \geq 1: \nu(x-y) \leq C \, \nu(x), \ |y| \leq s_1, \ s_2 \leq |x| < s_3 \right\}.
\end{align}
Using (A1.1)-(A1.2) we see that $K^X_2(s_1,s_2,s_3)$ is well-defined and a non-decreasing function in  $s_1 \in (0,s_2)$,
for every fixed $0< s_2 < s_3$. Whenever $s_3=\infty$ and $s_2 \gg s_1 \gg 1$, which will be mostly the case considered below,
$1/K^X_2(s_1,s_2,s_3)$ can be interpreted to measure the rate of preference of a direct large jump from $x \in B(0,s_2)^c$ to $0$
over a jump from a point $z$ situated in the $s_1$-neighbourhood of $x$ to $0$. In our applications, the second case corresponds
to the situation when the process moving from $x$ to $0$ first fluctuates inside $B(x,s_1)$ and then makes one large jump to the
origin. On this account, we refer to $K^X_2$ as the inverse rate of preference of the scenario ``no small steps but direct large jump" over the scenario ``first small steps, then large jump".

\medskip
\item
Let assumption (A3) hold and define $K^X_3:(0,\infty) \to (0,\infty)$ by
\begin{align} \label{def:K_3}
K^X_3(s):= \sup_{x,y: \, |x-y| \geq s/8 } G_{B(0,s)}(x,y), \quad s > 0.
\end{align}
\end{enumerate}

\medskip
These parameters will be relevant through their behaviour in some jump scenarios which we discuss next. First consider the
following asymptotic properties involving $K^X_1$ and $K^X_2$:
\begin{align} \label{eq:intr_killing}
\mbox{there exists \ $\kappa_1 \geq 2$ \ such that \ $\lim_{s \to \infty} K^X_1(\kappa_1 s) \,
K^X_2(s,\kappa_1 s,\infty) = 0$}
\end{align}
and
\begin{align} \label{eq:unif_bdd}
\mbox{there exists \ $\kappa_2 < \infty$ \ such that for all \ $s_1 \geq 1$ \ we have \
$\limsup_{s \to \infty}K^X_2(s_1, s,\infty) \leq \kappa_2$}.
\end{align}
Due to the roles played by conditions \eqref{eq:intr_killing}-\eqref{eq:unif_bdd} in
Sections \ref{sec:harm}-\ref{sec:eig_decay} below, we think
of them as expressing the capacity of responsiveness to perturbation of the process $\pro X$.
Note that
\begin{equation}
\label{K1tozero}
\lim_{s \to \infty} K^X_1(s) = 0
\end{equation}
is necessary but not sufficient for \eqref{eq:intr_killing}, even if \eqref{eq:unif_bdd} holds.

To understand what \eqref{eq:intr_killing} means, first notice that by the definition of $K_1^X$ it follows that
$$
\int_{|x-z|   >s_2 \atop |z|>s_2} \nu(x-z)\nu(z)dz \leq K^X_1(s_2) \nu(x), \quad \mbox{for all $x$ such that $|x|>s_2,$}
$$
and using the definition of $K_2^X$ we have
$$
\nu(x) \leq K^X_2(s_1, s_2, \infty)\nu(w), \quad \mbox{for all $w$ such that $|w|>s_2, \ |x-w|<s_1.$}
$$
Thus $1/(K^X_1(s_2) K^X_2(s_1, s_2, \infty))$ measures the rate of preference of a single large jump from $w$ to the origin
(``direct single large jump" scenario) over two large jumps from $x \in B(w,s_1)$ to $0$ (``first small steps, then two
large jumps" scenario). This means that property \eqref{eq:intr_killing} corresponds to the situation when the first
scenario outdoes the second at a rate which improves when the length of the long jumps increases appropriately with
the scale of the smaller fluctuations.
Secondly, observe that \eqref{eq:unif_bdd} says that the ``no small steps but direct large jump" scenario above
outdoes the ``first small steps, then large jump" scenario at a rate $1/\kappa_2$, no matter how large the smaller
fluctuations are (i.e., there is always a suitably larger jump for which the former is preferred to occur).

Next we show that $K^X_1$ dominates the tail of the corresponding L\'evy measure.

\begin{lemma} \label{lem:tail_dom}
Under Assumption (A1) we have
\begin{align} \label{eq:tail_dom}
K^X_1(s) \geq \frac{\nu \big(B(0, s)^c \big)}{2 \, C_5^4}, \quad s \geq 1.
\end{align}
\end{lemma}

\begin{proof}
Let $s \geq 1$ and denote $x_{s,n} = (2s+n, 0, ..., 0)$, $n \in \N$. By the definition of $K^X_1$ and the monotonicity
of $g$, for every $n \in \N$ we have
\begin{align*}
C_5 K^X_1(s)  g(2s+n) & \geq K^X_1(s) \nu(x_{s,n}) \geq \int_{|x_{s,n}-y|>s \atop |y| > s} \nu(x_{s,n}-y) \nu(y) dy \\
& \geq \frac{1}{C_5^2} \int_{|x_{s,n}-y|>s \atop |y| > s} g(|x_{s,n}-y|) g(|y|) dy
\geq  \frac{ g(2s+n)}{C_5^2} \int_{s < |x_{s,n}-y| < 2s+n  \atop s < |y| < s +n}  g(|y|) dy.
\end{align*}
For every $y=(y_1,...,y_d)$ such that $y_1 >0$ we have $|y-x_{s,n}| < 2s+n$ for sufficiently large $n \in \N$. Hence, for every $y \in \R^d$,  
$$\1_{\left\{z: \, s < |x_{s,n}-z| < 2s+n, \, s < |z| < s +n\right\}}(y) \to \1_{\left\{z: \, |z| > s, \, z_1>0 \right\}}(y) \quad \text{as} \  n \to \infty.$$ 
Using Fatou's lemma and radial symmetry, this implies
$$
C^3_5 K^X_1(s) \geq \liminf_{n \to \infty} \int_{s < |x_{s,n}-y| < 2s+n  \atop s < |y| < s + n}  g(|y|) dy \geq \int_{|y| > s, \, y_1>0}  g(|y|) dy = \frac{1}{2} \int_{|y| > s}  g(|y|) dy.
$$
Applying (A1) again, we obtain the claimed bound.
\end{proof}

Finally, consider $K^X_3$. We will require it to satisfy
\begin{align} \label{eq:Green_bdd}
\sup_{s \geq 1} \Big[K^X_3(s) \, \Psi(1/s) \, s^d \Big] < \infty.
\end{align}
This is a regularity condition known to hold for a large class of processes, in particular, for isotropic unimodal L\'evy processes
(i.e., $A \equiv a \Id$ for some $a \geq 0$ and $\nu(x)$ is a non-increasing radial function) and $d \geq 3$, see
\cite[Th. 3]{bib:Grz}. It is reasonable to conjecture that actually \eqref{eq:Green_bdd} is valid in a generality which
covers all the cases considered in the present paper, but a general argument does not seem to be available. While this
condition is not decisive for our results below, we use it as a convenient technical assumption. The following lemma gives
a sufficient condition.

For a continuous non-decreasing function $\Phi:[0,\infty) \to [0,\infty)$ such that $\Phi(0)=0$
and $\lim_{r \to \infty} \Phi(r) = \infty$ we denote
$$
\Phi^{-1}(s)=\sup\{r \geq 0: \Phi(r)=s\} \quad \text{and} \quad \Phi_{*}^{-1}(s)=\inf\{r \geq 0: \Phi(r)=s\}, \quad s \geq 0,
$$
so that $\Phi(\Phi^{-1}(s))= \Phi(\Phi_{*}^{-1}(s))=s$, $\Phi^{-1}(\Phi(s))\geq s$ and $\Phi_{*}^{-1}(\Phi(s)) \leq s$ for $s \geq 0$.

\begin{lemma} \label{lem:ver_cond4}
Let $\pro X$ be a L\'evy process determined by the L\'evy-Khinchin exponent $\psi$ as in \eqref{eq:Lchexp} such that
\eqref{eq:nuinf} holds. Suppose, moreover, that there exists a continuous non-decreasing function
$\Phi:[0,\infty) \to [0,\infty)$ such that $\Phi(0)=0$ and $\lim_{r \to \infty} \Phi(r) = \infty$,
with the doubling property $\Phi(2s) \leq C\Phi(s)$, $s>0$, and a number $\Theta \geq 0$ for which
\begin{align} \label{eq:suf0}
\sup_{|x| \geq r} p(t,x) \leq C_{9} \, t\,  \left(\frac{\Phi(1/r)}{r^d} + \Theta\right), \quad t>0, \ \ \ r \geq 1,
\end{align}
and
\begin{align} \label{eq:suf2}
p(t,0) = \int_{\R^d} e^{-t \psi(z)} dz \leq C_{10} \left(\Phi_{*}^{-1}\left(\frac{1}{t}\right)\right)^d, \quad t \geq t_0,
\end{align}
for some $t_0>0$ and constants $C_{9}$, $C_{10}$. Then there exists $r_0 \geq 1$ such
that
$$
K^X_3(r) \leq 8^d e C^3 C_{9} \frac{1}{\Phi(1/r) r^d} + e C_9 \frac{\Theta}{\Phi(1/r)^2} + 4 C_1 C_2 C_4 C_{10} \frac{1}{\Psi(1/r)r^d},
\quad r \geq r_0.
$$
In particular, if this is true with $\Phi = \Psi$ and $\Theta = 0$, then \eqref{eq:Green_bdd} holds.
\end{lemma}

\begin{proof}
Starting from the general estimate \cite[Prop. 2.3]{bib:BKK}, for $t \geq t_0$ and $\eta >0$ we have
\begin{align} \label{eq:est3}
\sup_{\left\{(x,y): \, |x-y| \geq r/8 \right\}} G_{B(0,r)}(x,y) \leq e^{\eta t} \sup_{|z| \geq r/8} \int_0^{\infty} e^{-\eta s} p(s,z)ds + \ex^0[\tau_{B(0,2r)}] \, \int_{\R^d} e^{-t \psi(z)} dz.
\end{align}
By \eqref{eq:suf0} it follows that
$$
 \sup_{|z| \geq r/8} \int_0^{\infty} e^{-\eta s} p(s,z) ds \leq C_{9} \left(8^d
 \frac{\Phi(8/r)}{\eta^2 \, r^d} + \frac{\Theta}{\eta^2} \right)
, \quad r \geq 1.
$$
Taking $\eta=1/t$ with $t=1/\Phi(1/r)$ in \eqref{eq:est3}, and using \eqref{eq:suf2}, \eqref{eq:gen_est_tau} and the doubling property of $\Phi$, we obtain the claimed inequality.
\end{proof}
\noindent
Below we will mainly consider condition \eqref{eq:suf0} in the form
\begin{align} \label{eq:suf1}
\sup_{|x| \geq r} p(t,x) \leq C_{9}\frac{t \, \Psi(1/r)}{r^d}, \quad t>0, \ \ \ r \geq 1,
\end{align}
i.e., when $\Phi = \Psi$ and $\Theta = 0$. Moreover, we will often use the property that if $\Psi(r) \asymp r^{\alpha}$,
$r \in [0,r_0]$, for some $\alpha >0$ and $r_0>0$, then $\Psi^{-1}(r) \asymp \Psi_{*}^{-1}(r) \asymp r^{1/\alpha}$,
$r \in [0, \Psi(r_0)]$. Conditions \eqref{eq:suf0}-\eqref{eq:suf2} and \eqref{eq:suf1} will be discussed and illustrated
on some specific cases in the next sections.

\subsection{Feynman-Kac semigroup and non-local Schr\"odinger operator}
\noindent
We now give the class of potentials which will be used in this paper.
\begin{definition}[\textbf{$X$-Kato class}]
{\rm
We say that the Borel function $V: \R^d \to \R$ called \emph{potential} belongs to \emph{Kato-class} $\cK^X$ associated
with the L\'evy process $\pro X$ if it satisfies
\begin{align}
\label{eq:Katoclass}
\lim_{t \downarrow 0} \sup_{x \in \R^d} \ex^x \left[\int_0^t |V(X_s)| ds\right] = 0.
\end{align}
Also, we say that $V$ is an \emph{$X$-Kato decomposable potential}, denoted $V \in \cK^X_{\pm}$, whenever
$$
V=V_+-V_-, \quad \text{with} \quad V_- \in \cK^X \quad \text{and} \quad V_+ \in \cK^X_{\loc},
$$
where $V_+$, $V_-$ denote the positive and negative parts of $V$, respectively, and where $V_+ \in \cK^X_{\loc}$
means that $V_+ 1_B \in \cK^X$ for all compact sets $B \subset \R^d$.
\label{Xkato}
}
\end{definition}
\noindent
For simplicity, in what follows we refer to $X$-Kato decomposable potentials as \emph{$X$-Kato class potentials}.
It is straightforward to see that $L^{\infty}_{\loc}(\Rd) \subset \cK_{\loc}^X$. Moreover, by stochastic continuity
of $\pro X$ also $\cK_{\loc}^X \subset L^1_{\loc}(\R^d)$, and thus an $X$-Kato class potential is locally
absolutely integrable. Note that condition \eqref{eq:Katoclass} allows local singularities of $V$. For specific
processes $\pro X$ the definition of $X$-Kato class can be explicitly reformulated in an analytic way in terms
of the kernel $p(t,x)$ restricted to small $t$ and small $x$. It is shown in \cite[Cor. 1.3]{bib:GrzSz} that
\eqref{eq:Katoclass} is equivalent with
\begin{align} \label{eq:Kato_new}
\lim_{t \to 0^{+}} \sup_{x \in \R^d} \int_0^t \int_{B(x,t)} p(s,x-y)|V(y)| dy ds = 0.
\end{align}

Define
$$
T_t f(x) = \ex^x\left[e^{-\int_0^t V(X_s) ds} f(X_t)\right], \quad f \in L^2(\R^d), \ t>0.
$$
By standard arguments based on Khasminskii's Lemma, see \cite[Cor.Prop.3.8]{bib:CZ},\cite[Lem.3.37-3.38]{bib:LHB},
for an $X$-Kato class potential $V$ it follows that there exist constants $C_{11}=C_{11}(X,V)$ and $C_{12} =
C_{12}(X,V)$ such that
\begin{align}
\label{eq:khas}
\sup_{x \in \R^d} \ex^x\left[e^{-\int_0^t V(X_s)ds}\right] \leq \sup_{x \in \R^d}
\ex^x\left[e^{\int_0^t V_-(X_s)ds}\right] \leq C_{11} e^{C_{12}t}, \quad t>0.
\end{align}
Using the Markov property and stochastic continuity of the process it can be shown that
$\{T_t: t\geq 0\}$ is a strongly continuous semigroup of symmetric operators on $L^2(\R^d)$, which we call the
\emph{Feynman-Kac semigroup} associated with the process $\pro X$ and potential $V$. In particular, by the
Hille-Yoshida theorem there exists a self-adjoint operator $H$, bounded from below, such that $e^{-t H} = T_t$.
We call the operator $H$ a \emph{non-local Schr\"odinger operator} based on the infinitesimal generator $L$ of
the process $\pro X$. Since any $X$-Kato class potential is relatively form bounded with respect to the ``free Hamiltonian"
$H_0=-L$ with relative bound less than 1, we have $H=H_0+V$, where the latter operator is defined in form sense \cite[Ch. 2]{bib:DC}. For subordinate Brownian motions,
$H$ becomes a non-local
Schr\"odinger operator with a Bernstein function of the Laplacian studied in \cite{bib:HIL12}, i.e., has the form
$\phi(-\Delta) + V$, where $\phi$ is the Laplace exponent of the corresponding subordinator.

We now summarize the basic properties of the operators $T_t$ which will be useful below.

\begin{lemma}
\label{lm:semprop}
Let $\pro X$ be a symmetric L\'evy process with L\'evy-Khinchin exponent satisfying \eqref{eq:Lchexp}-\eqref{eq:nuinf}
such that Assumption (A2) holds, and let $V$ be an $X$-Kato class potential. Then the following properties hold:
\begin{itemize}
\item[(1)]
For all $t>0$, every $T_t$ is a bounded operator on every $L^p(\R^d)$ space, $1 \leq p \leq \infty$. The operators
$T_t: L^p(\R^d) \to L^p(\R^d)$ for $1 \leq p \leq \infty$, $t > 0$, and $T_t: L^p(\R^d) \to L^{\infty}(\R^d)$
for $1 < p \leq \infty$, $t \geq t_{\rm b}$, and $T_t: L^1(\R^d) \to L^{\infty}(\R^d)$ for $t \geq 2t_{\rm b}$
are bounded, with some $t_{\rm b}    > 0$.
\vspace{0.1cm}
\item[(2)]
For all $t \geq 2t_{\rm b}$, $T_t$ has a bounded measurable kernel $u(t, x, y)$ symmetric in $x$ and $y$,
i.e., $T_t f(x) = \int_{\R^d} u(t,x,y)f(y) dy$, for all $f \in L^p(\R^d)$ and $1 \leq p \leq \infty$.
\vspace{0.1cm}
\item[(3)]
For all $t>0$ and $f \in L^{\infty}(\R^d)$, $T_t f$ is a bounded continuous function.
\vspace{0.1cm}
\item[(4)]
For all $t \geq 2t_{\rm b}$ the operators $T_t$ are positivity improving, i.e., $T_t f(x) > 0$ for all
$x \in \R^d$ and $f \in L^2(\R^d)$ such that $f \geq 0$ and $f \neq 0$ a.e.
\end{itemize}
\end{lemma}

\noindent
The above properties can be established by standard arguments, see \cite[Sect. 3.2]{bib:CZ}. Note that we do not
assume that $p(t,x)$ is bounded for all $t>0$, and thus in general the operators $T_t: L^p(\R^d) \to L^{\infty}(\R^d)$
need not be bounded for $t<t_{\rm b}$.

Related to the Feynman-Kac semigroup, we also define the potential operator by
\begin{align*}
G^V f(x) = \int_0^\infty T_t f(x) dt = \ex^x \left[\int_0^\infty e^{-\int_0^t V(X_s)ds} f(X_t) dt \right],
\end{align*}
for non-negative or bounded Borel functions $f$ on $\R^d$. Recall that $\tau_D$ denotes the first exit time of
the process from domain $D$. Whenever $D \subset \R^d$ is an open set and $f$ is a non-negative or bounded Borel
function on $\R^d$, it follows by the strong Markov property of the process that for every $x \in D$
\begin{equation}
\label{eq:pot1}
\begin{split}
G^V f(x)
& =
\ex^x\left[\int_0^{\tau_D} e^{-\int_0^t V(X_s)ds} \, f(X_t) dt \right] + \ex^x\left[\tau_D < \infty;
e^{-\int_0^{\tau_D} V(X_s)ds} \, G^V f(X_{\tau_D})\right].
\end{split}
\end{equation}
For background on potential theory we refer to \cite{bib:CZ, bib:BlG, bib:Ber, bib:BB1, bib:CS, bib:BBKRSV, bib:BKK}.

\bigskip
\section{Estimates of harmonic functions}
\label{sec:harm}

\subsection{Lower bound for functions harmonic at infinity}

Using \eqref{eq:pot1} it will be seen that eigenfunctions for a given process and a given potential are comparable to specific harmonic functions; this will be explored to a large extent by deriving and using the representation (\ref{eq:eig1}) below. In this section first we develop some technical tools concerning harmonic functions.

Let $\eta  >0$. Recall that a non-negative Borel function $f$ on $\R^d$ is called $(X,\eta)$-harmonic in an open set
$D \subset \R^d$ if
\begin{align}
\label{def:harm}
f(x) & = \ex^x\left[\tau_U < \infty; e^{-\eta \tau_U} f(X_{\tau_U})\right], \quad x \in U,
\end{align}
for every open set $U$ with its closure $\overline{U}$ contained in $D$, and it is called regular $(X,\eta)$-harmonic
in $D$ if \eqref{def:harm} holds for $U=D$ (where $\tau_U$ is the first exit time from $U$). By the
strong Markov property every regular $(X,\eta)$-harmonic function in $D$ is $(X,\eta)$-harmonic in $D$. Below we mainly
consider the case when \eqref{def:harm} holds with $D = \overline B(0,r)^c$ for some $r   >0$. We refer to this property
as \emph{harmonicity at infinity}.

We also recall that when $D \subset \R^d$ is a bounded open domain, the following formula due to Ikeda and
Watanabe holds \cite[Th. 1]{bib:IW}: for every $\eta>0$ and every bounded or non-negative Borel function $f$ on
$\R^d$ such that $\dist(\supp f, D) >0$, we have
\begin{align}
\label{eq:IWF}
\ex^x\left[e^{-\eta \tau_D} f(X_{\tau_D})\right] = \int_D \int_0^{\infty} e^{-\eta t} p_D(t,x,y) dt \int_{D^c} f(z)
\nu(z-y) dzdy, \quad x \in D.
\end{align}

The next theorem is our first result in this section. It states that non-negative functions that are
$(X,\eta)$-harmonic at infinity are bounded from below by $\nu$, no matter how small $\eta$ is. On the other
hand, it says that when condition (A1.2), or at least the jump-paring property (A1.3), fails to hold, then
such functions will not be dominated by $\nu$ at infinity even for large $\eta$. This means that it is
reasonable to ask which properties of jump-paring L\'evy processes will guarantee that the functions
$(X,\eta)$-harmonic at infinity are comparable to $\nu$ at least far away from the origin.

\begin{theorem}
\label{prop:lowerboundharm}
Let $\eta > 0$ and $\pro X$ be a symmetric L\'evy process with L\'evy-Khinchin exponent satisfying
\eqref{eq:Lchexp}-\eqref{eq:nuinf}. Let $r>0$ and $f$ be a non-negative $(X,\eta)$-harmonic function in
$\overline B(0,r)^c$ such that $\int_{B(0,r)} f(z) dz > 0$. Then we have the following:
\begin{itemize}
\item[(1)]
If (A1.1)-(A1.2) hold, then
$$
f(x) \geq
\left( \frac{1-e^{-\eta}}{C_5^2 C_6^{\left\lceil r\right\rceil+1} \eta} \, \pr^0(\tau_{B(0,1)} > 1)
\int_{B(0,r)} f(z)dz \right) \, \nu(x), \quad |x| > r + 1.
$$
\item[(2)]
Let (A1.1) hold and suppose $\inf_{|y| \leq (r \vee 2)+1} f(y) > 0$. Consider the following disjoint cases:
\begin{itemize}
\item[(i)]
(A1.2) holds and (A1.3) does not hold.
\item[(ii)]
(A1.2) does not hold (and hence (A1.3) fails to hold).
\end{itemize}
Then in either of cases (i) and (ii) we have $\limsup_{|x| \to \infty} \frac{f(x)}{\nu(x)} = \infty$.
\end{itemize}
\end{theorem}

\begin{proof}
First note that by $(X,\eta)$-harmonicity of $f$ in $\overline B(0,r)^c$ we have $f(x) = \ex^x[e^{-\eta \tau_{B(x,1)}}
f(X_{\tau_{B(x,1)}})]$, for every $|x| > r+1$. Thus by the Ikeda-Watanabe formula \eqref{eq:IWF}, for every
$|x| > r+1$
\begin{align*}
f(x) & \geq \int_0^{\infty} e^{-\eta t} \pr^x(\tau_{B(x,1)} > t) dt \,  \inf_{y \in B(x,1)} \int_{B(0,r)} f(z) \nu(z-y) dz \\
& \geq
\frac{1}{C_5} \left(\int_0^1 e^{-\eta t} dt \, \pr^0(\tau_{B(0,1)} > 1) \int_{B(0,r)} f(z) dz\right) g(|x|+1+r) \\
& \geq
\left(\frac{1-e^{-\eta}}{C_5^2 C_6^{\left\lceil r\right\rceil+1} \eta} \, \pr^0(\tau_{B(0,1)}   > 1) \int_{B(0,r)} f(z)dz \right)\,
\nu(x),
\end{align*}
which proves (1). We now show (2). When (A1.2) fails to hold (case (ii)), there exists a sequence $\seq r$ such that
$g(r_n-1) \geq n g(r_n)$. Notice that necessarily $r_n \to \infty$ as $n \to \infty$. Let $x_n=(r_n,0,...,0)$. With this,
by following the argument as in (1) above, we get
\begin{align*}
f(x_n) & \geq \int_0^{\infty} e^{-\eta t}  \int_{B(x_n,1)} e^{-\eta t} p_{B(x_n,1)}(t, x_n, z) dt \,  \int_{B(0,(r \vee 2)+1)}
f(y) \nu(z-y) dydz \\
& \geq
\frac{1}{C_5} \left(\int_0^1 e^{-\eta t} dt \, \pr^0(\tau_{B(0,1)} > 1) \int_{B((5x_n/2r_n) ,1/2)} f(y) dy\right) g(r_n-1) \\
& \geq
n \, \left(\frac{1-e^{-\eta}}{C_5  \eta} \, \pr^0(\tau_{B(0,1)}   > 1) \, \inf_{|y| \leq (r \vee 2)+1} f(y) |B(0,1/2)| \right)\, g(r_n),
\end{align*}
for sufficiently large $n$. From this we easily see that $\lim_{n \to \infty}f(x_n)/\nu(x_n) = \infty$.

Consider now the case (i) and suppose that (A1.2) holds while (A1.3) does not. If (A1.3) breaks down, then there exists a
divergent sequence $\seq s$ such that for sufficiently large $n$
\begin{align}\label{eq:aux_seq}
\int_{|y-x_n| > 1 \atop |y|>1} g(|x_n-y|) g(|y|) dy \geq n \, g(|x_n|), \quad n \in \N, \quad \text{where} \quad x_n=(s_n,0,...,0).
\end{align}
By the facts $g(|y|) \leq g(1) < \infty$, $|y| \geq 1$, and $\inf_{|y| \leq r+1} f(y)    > 0$, there exists $c=c(r)$ such that $g(|y|)
\leq c f(y)$, for $1 \leq |y| \leq r+1$. This implies jointly with (1) that
$$
f(y) \geq c_1 g(|y|) , \quad |y| \geq 1, \quad \text{for some} \; \; c_1 > 0.
$$
Proceeding now in the same way as in (1), by making use of \eqref{eq:aux_seq} and (A1.2) we get
\begin{align*}
f(x_n) & \geq \frac{c_1}{C_5} \int_0^{\infty} \int_{B(x_n,1/2)} e^{-\eta t} p_{B(x_n,1/2)}(t, x_n, z) dt \,  \int_{|y-x_n| > 1/2
\atop |y|>1} g(|z-y|) g(|y|) dy dz \\
& \geq
 \frac{c_2}{C_5} \int_0^1 e^{-\eta t} dt \, \pr^0(\tau_{B(0,1/2)} > 1) \, \int_{|y-x_n|   > 1 \atop |y|>1} g(|x_n-y|) g(|y|) dy dz \\
& \geq n \left( \frac{c_2}{C_5} \int_0^1 e^{-\eta t} dt \, \pr^0(\tau_{B(0,1/2)} > 1) \right) g(|x_n|),
\end{align*}
for sufficiently large $n$. Similarly as above, this completes the proof in the case (i).
\end{proof}

\subsection{Uniform estimate for suprema of functions harmonic in balls}
\label{subsec:harm}
In this subsection we derive a uniform upper bound for functions that are regular $(X,\eta)$-harmonic in large balls,
which will be a basic technical tool in what follows. It can be obtained as an adaptation of the strong estimates of \cite{bib:BKK} to our framework.

For $s_1\geq 1$ and $s_2 \geq 2s_1$ define
\begin{align*}
h_1(X,s_1,s_2) = K^X_2(s_1, s_2,\infty) \left[C_3\left(X,\frac{s_1}{16}\right) \Big(C_{13}(X,s_1)
\, |B(0,s_1)|  + \ex^0 [\tau_{B(0,2s_1)}]\Big) + 1 \right]
\end{align*}
and
\begin{align*}
h_2(X, s_1) = C_3\left(X,\frac{s_1}{16}\right) \Bigg[C_3\left(X,s_1\right) C_{13}(X,s_1) +
\ex^0 [\tau_{B(0,2s_1)}] \sup_{|y| \geq \frac{s_1}{4}} \nu(y)\Bigg] + \sup_{|y| \geq \frac{s_1}{16}} \nu(y),
\end{align*}
where
$$
C_{13}(X,s_1) := K^X_3(s_1) + \frac{\ex^0 [\tau_{B(0,2s_1)}]}{\left|B(0,\frac{s_1}{4})\right|}
\left(K^X_2\left(\frac{s_1}{4}, \frac{s_1}{2}, s_1\right)\right)^2.
$$
It is of key importance that neither $h_1$ nor $h_2$ depends on $x$, $\eta$ and the function $f$. With the above
notation we have the following.

\begin{lemma}
\label{lm:bhi}
Let $\pro X$ be a L\'evy process with L\'evy-Khinchin exponent $\psi$ given by \eqref{eq:Lchexp}-\eqref{eq:nuinf} such
that Assumptions (A1.1)-(A1.2) and (A3) hold, and let $s_1 \geq 1$ and $s_2 \geq 2s_1$. Then for every $\eta  >0$ and
every non-negative function $f$ on $\R^d$ which is regular $(X,\eta)$-harmonic in a ball $B(x,s_1)$, $x \in \R^d$, we have
\begin{align}
\label{eq:harmest}
f(y) \leq \frac{1}{\eta} \, \left(h_1(X,s_1, s_2) \int_{|z-x| > s_2}f(z) \nu(z-x)dz
+ h_2(X, s_1) \int_{\frac{s_1}{8} < |x-z| \leq s_2} f(z)dz \right) , \quad |y-x|< \frac{s_1}{32}.
\end{align}
\end{lemma}

\begin{proof}
The claimed bound is a version of the upper estimate in (3.3) of \cite[Lem. 3.2]{bib:BKK}, but the constant in that bound
is not suitable for our purposes here and \eqref{eq:harmest} does not follow from the statement in the cited paper. However, the required form of constants can be obtained from the original statement with some extra work, and in this proof we keep to the notation of \cite{bib:BKK} for the reader's convenience.

Note that by conditions \eqref{eq:nuinf}, (A1.1)-(A1.2) and (A3) above all of the Assumptions A, B, C, D in \cite{bib:BKK}
hold. Let $s_1 \geq 1$ and $s_2 \geq 2s_1$. First we show that a version of the estimate \cite[(3.7)]{bib:BKK} holds with
$R=s_1$, $q=s_1/2$, $p=s_1/4$ and $r=s_1/8$, i.e.,
\begin{align} \label{eq:harm_est}
f(y) \leq \int_{|z-x| \geq \frac{s_1}{2}} f(z) \pi_{s_1,s_2}(z-x) dz, \quad |y-x| < \frac{s_1}{8},
\end{align}
where
$$
\pi_{s_1,s_2}(z-x) = \left\{
\begin{array}{lll}
C_3(X,s_1) C_{13}(X,s_1) + \ex^0 [\tau_{B(0,2s_1)}] \, \sup_{|y| \geq \frac{s_1}{4}} \nu(y)
\ \ \ \ \ \ \ \ \ \ \ \ \ \mbox{for} & \frac{s_1}{2} \leq |z-x| \leq s_2, &  \vspace{0.2cm} \\
K^X_2(s_1,s_2,\infty) \Big( C_{13}(X,s_1) |B(0,s_1)| + \ex^0 [\tau_{B(0,2s_1)}] \Big ) \, \nu(z-x)
& \ \mbox{for} \ \ |z-x|  > s_2.&
\end{array}\right.
$$
This can be seen by following through the argument in \cite[Th. 3.4]{bib:BKK}. First observe that the constant $c_{(2.7)}(x,p,q)$ can
be ``localized'' in space, i.e., its value actually depends on the position $z$ appearing in $\nu(z-x)$ and $\nu(z-y)$ when $|x-y| < p = s_1/4$. Indeed, we have
\begin{align} \label{eq:c27}
c_{(2.7)}(x,p,q) = \left\{
\begin{array}{lrl}
K^X_2(s_1/4, s_1/2, s_1)  & \mbox{  for } & \frac{s_1}{2} \leq |z-x| \leq s_1, \vspace{0.2cm} \\
K^X_2(s_1/4,s_1, s_2)   & \mbox{  for } & s_1 < |z-x| \leq s_2, \vspace{0.2cm} \\
K^X_2(s_1/4, s_2, \infty)   & \mbox{  for  } & |z-x|    > s_2.
\end{array}\right.
\end{align}
Moreover, recall that $K^X_2(v, s_2, \infty)$ is non-decreasing in $v \in (0, s_2)$ (we emphasize that in
our space-homogeneous case none of the constants depends on $x$). Also, we see that in our setting
$c_{(2.9)}(x,R) \leq \ex^0[\tau_{B(0,2s_1)}]$ and $c_{(2.10)}(x,r,p,R) \leq K_3^X(s_1)$.
By this and \eqref{eq:c27} we can directly check that in \cite[Lem. 4.5]{bib:BKK} we have
$c_{(4.17)}(x,r,p,q,R) \leq C_{13}(X,s_1)$, i.e., the constant $c_{(2.7)}(x,p,q)$ appearing in the estimates
is equal to $K^X_2(s_1/4, s_1/2, s_1)$. With this, we can now verify that the statement of \cite[Lem. 4.9]{bib:BKK}
stays valid with the kernel $\tilde{\pi}_{\psi}(z)$ replaced by
$$
\tilde{\pi}_{s_1,s_2}(z-x) = \left\{
\begin{array}{lll}
C_{13}(X,s_1) \delta + \ex^0 [\tau_{B(0,2s_1)}] \, \sup_{|w| \geq \frac{s_1}{4}} \nu(w)
\ \ \ \ \ \ \ \ \ \ \ \ \ \mbox{for} & \frac{s_1}{2} \leq |z-x| \leq s_2, &  \vspace{0.2cm} \\
K^X_2(s_1,s_2,\infty) \Big( C_{13}(X,s_1) |B(0,s_1)| + \ex^0 [\tau_{B(0,2s_1)}] \Big ) \, \nu(z-x)
& \ \mbox{for} \ \ |z-x| > s_2,&
\end{array}\right.
$$
in \cite[(4.21)]{bib:BKK}.
To do that, recall the specific choices $R=s_1$, $q=s_1/2$, $p=s_1/4$, $r=s_1/8$, and observe that \cite[(4.22)]{bib:BKK} holds with the constant $C_{13}(X,s_1) \, \delta$. We can now continue similarly
as in the second part of the proof of \cite[Lem.4.9]{bib:BKK}.
For $|w-x| \leq s_1/8$ and $z  \in V^c$ such that
$s_1/2 \leq |z-x| \leq s_2$ we have
$$
\int_{V \cap B(x,s_1/4)^c} G_{\psi}(w,y) \nu(y-z) dy = \int_{V \cap B(x,s_1/4)^c} G_{\psi}(y,w) \nu(y-z) dy \leq C_{13}(X,s_1) \, \delta
$$
and
$$
\int_{B(x,s_1/4)} G_{\psi}(w,y) \nu(y-z) dy \leq \sup_{|v| \geq \frac{s_1}{4}} \nu(v) \, \ex^w [\tau_{B(x,s_1)}]
\leq \sup_{|v| \geq \frac{s_1}{4}} \nu(v) \, \ex^0 [\tau_{B(0,2s_1)}].
$$
(Since we use here the original notation, $V$ now means the set defined in \cite[(4.2)]{bib:BKK}, and $G_{\psi}(w,y)$ is the potential kernel as in \cite[p.492]{bib:BKK}.)
Moreover, for $|z-x| > s_2$, in which case necessarily $z  \in V^c$, we get
\begin{align*}
\int_{V \cap B(x,s_1/4)^c} G_{\psi}(w,y) \nu(y-z) dy & \leq C_{13}(X,s_1) \nu(B(x,s_1)-z) \\ & \leq C_{13}(X,s_1) K^X_2(s_1,s_2,\infty) |B(0,s_1)| \, \nu(x-z)
\end{align*}
and
\begin{align*}
\int_{B(x,s_1/4)} G_{\psi}(w,y) \nu(y,z) dy & \leq K^X_2(s_1,s_2,\infty) \, \ex^w [\tau_{B(x,s_1)}]  \, \nu(x-z) \\
& \leq K^X_2(s_1,s_2,\infty) \, \ex^0 [\tau_{B(0,2s_1)}]  \, \nu(x-z).
\end{align*}
Combining these estimates, we conclude similarly as in \cite[Lem. 4.9]{bib:BKK} that the claimed bound holds with
the kernel $\tilde{\pi}_{\psi}(z)$ replaced by $\tilde{\pi}_{s_1,s_2}(z-x)$. Since $\varrho = \varrho\left(\overline B(x,q),B(x,R))\right) \leq C_3(X,s_1)$, the proof of \eqref{eq:harm_est} can be completed in the same way as in \cite[Th. 3.4]{bib:BKK}.

To complete the proof of the lemma, we follow the argument leading from (a) to (b) in the proof of \cite[Lem. 3.2]{bib:BKK}. By regular
$(X,\eta)$-harmonicity of $f$ we get
\begin{align*}
f(y) & = \ex^y\left[e^{-\eta \tau_{B\left(x,\frac{s_1}{16}\right)}} f(X_{\tau_{B\left(x,\frac{s_1}{16}\right)}});
 X_{\tau_{B\left(x,\frac{s_1}{16}\right)}}
\in B\left(x,\frac{s_1}{8}\right) \right] \\
& \ \ \ \ + \ex^y\left[e^{-\eta \tau_{B\left(x,\frac{s_1}{16}\right)}}
f(X_{\tau_{B\left(x,\frac{s_1}{16}\right)}});
X_{\tau_{B\left(x,\frac{s_1}{16}\right)}} \in B\left(x,\frac{s_1}{8}\right)^c \right] = {\rm I} + {\rm II},
\end{align*}
whenever $|y-x| < s_1/32$. By the Ikeda-Watanabe formula \eqref{eq:IWF} we furthermore have
\begin{align*}
{\rm II} & \leq \ \ex^y \left[\int_0^{\tau_{B\left(x,\frac{s_1}{16}\right)}} e^{-\eta t} dt \right]
\Bigg(\sup_{|w|   >\frac{s_1}{16}} \nu(w) \int_{\frac{s_1}{8} \leq |z-x| \leq s_2} f(z) dz \\
& \qquad \qquad \qquad \qquad
+ K^X_2\left(\frac{s_1}{16}, s_2, \infty \right) \int_{|z-x| > s_2} f(z) \nu(z-x) dz \Bigg)\\
& \leq
\frac{1}{\eta} \Bigg(\sup_{|w|>\frac{s_1}{16}} \nu(w) \int_{\frac{s_1}{8} \leq |z-x| \leq s_2} f(z) dz
+ K^X_2\left(s_1, s_2 , \infty \right) \int_{|z-x| > s_2} f(z) \nu(z-x) dz \Bigg).
\end{align*}
Similarly, by the version of \cite[Lem. 3.1]{bib:BKK}, see also \cite[Ex. 5.9]{bib:BKK}, for the subprocess of
$\pro X$ corresponding to the multiplicative functional $M_t=e^{-\eta t}$ and \eqref{eq:harm_est}, we obtain
\begin{align*}
{\rm I}
& \leq
\ \ex^y \left[e^{-\eta \tau_{B\left(x,\frac{s_1}{16}\right)}};
X_{\tau_{B\left(x,\frac{s_1}{16}\right)}} \in B\left(x,\frac{s_1}{8}\right) \right] \, \sup_{|y-x| <
\frac{s_1}{8}} f(y) \\ & \leq \frac{C_3\left(X, \frac{s_1}{16}\right)}{\eta} \int_{|z-x|    > \frac{s_1}{2}} f(z)
\pi_{s_1,s_2}(z-x) dz.
\end{align*}
By putting together the estimates of I and II above \eqref{eq:harmest} follows.
\end{proof}

\begin{remark}
{\rm
The estimate \eqref{eq:harmest} is somewhat laborious (with no apparent ways to simplify the argument), however, it has the correct structure required by the applications of Lemma \ref{lm:bhi} in the next two subsections. In Lemma \ref{lem:defic1} following below, \eqref{eq:harmest} will be iterated infinitely many times resulting in a self-improving estimate which leads to an upper bound of harmonic functions controlled by $\nu$ alone. To realise this for arbitrarily small $\eta >0$, we need to ensure that both prefactors $h_1, h_2$ in \eqref{eq:harmest} are small enough. This requires to work with a sufficiently large domain of harmonicity (large $s_1 >0$) and introduce the additional control parameters $r_2 \gg r_1$ and $r_3 > r_2$. Another significant difference between \eqref{eq:harmest} and the original bound in \cite{bib:BKK} is that our function $K_2^X$ appearing in the expression of $C_{13}$
depends only on the form of $\nu(y)$ for $s_1/4 \leq |y| \leq s_1 + s_1/4$ but not for $|y| \geq r_2$. The use of this construction will get fully transparent in the proof of our main technical result in Theorem \ref{thm:defic3}, where we deal with L\'evy measures fast decaying at infinity.
}
\end{remark}

\subsection{Upper bound for functions harmonic at infinity}
The next lemma is the first key technical result of this paper and it will be fundamental for our investigations
below. Recall that the functions $K^X_1, K^X_2$ are defined in \eqref{def:K1}-\eqref{def:K2}.
\begin{lemma}
\label{lem:defic1}
Let $\eta>0$ and $\pro X$ be a L\'evy process with L\'evy-Khinchin exponent $\psi$ given by
\eqref{eq:Lchexp}-\eqref{eq:nuinf} such that Assumptions (A1)--(A3) hold. Moreover, suppose
that there exist $r_1 \geq 1$, $r_2 \geq 2r_1$ and $r_3 > r_2$ such that
\begin{align} \label{eq:cond1}
2 \, C^4_5 \, h_1(X, r_1, r_2) \, K^X_1(r_2) +
h_2(X, r_1) \, |B(0,r_2)| \, K^X_2(r_2, r_3, \infty)< \eta.
\end{align}
Then for every bounded function $f \geq 0$ which is $(X,\eta)$-harmonic in $\overline B(0,r)^c$ for some
$r   > 0$, we have
$$
f(x) \leq C_{14} \left\|f\right\|_{\infty} \nu(x), \quad |x| \geq R+1,
$$
with $R=R(X,\eta):= (r+r_1) \vee r_3$ and
$$
C_{14}=C_{14}(X,\eta):= \frac{C_5^2C_6^{\left\lceil R \right\rceil}(1+C_8)\left(h_1(X,r_1, r_2) +
h_2(X, r_1)\frac{1}{\inf_{|y| \leq r_2} \nu(y)}\right)|B(0,R)|}{\eta - h_1(X,r_1, r_2) \, K^X_1(r_2)-
h_2(X, r_1) |B(0,r_2)| K^X_2(r_2, r_3, \infty)}.
$$
\end{lemma}

\begin{proof}
Fix $\eta, r >0$, $r_1>1$, $r_2 \geq 2r_1$ and $r_3 > r_2$ as in the statement. Denote
$R:= (r+r_1) \vee r_3$. By $(X,\eta)$-harmonicity of the function $f$ in $\overline{B}(0,r)^c$ we have
that $f(y) = \ex^y [e^{-\eta \tau_{B(x,r_1)}} f(X_{\tau_{B(x,r_1)}}]$, $y \in B(x,r_1)$, whenever $|x|
> R$. Thus by Lemma \ref{lm:bhi} and \eqref{eq:b2}, for every $|x| \geq R+1$ we obtain
\begin{eqnarray*}
f(x)
& \leq &
\frac{1}{\eta} \, \left(h_1(X,r_1, r_2) \int_{|z-x| > r_2}f(z) \nu(z-x)dz +
h_2(X, r_1) \int_{\frac{r_1}{8} < |x-z| \leq r_2} f(z)dz \right) \nonumber \\
&=&
\frac{1}{\eta} \, \left(h_1(X,r_1, r_2) \int_{|z-x|    > r_2 \atop |z| \leq R+1 }f(z)
\nu(z-x)dz + h_2(X, r_1) \int_{\frac{r_1}{8} < |x-z| \leq r_2 \atop |z| \leq R+1} f(z)dz \right)
\nonumber \\
& \qquad + &
\frac{1}{\eta} \, \left(h_1(X,r_1, r_2) \int_{|z-x|    > r_2 \atop |z| > R+1 }f(z) \nu(z-x)dz +
h_2(X, r_1) \int_{\frac{r_1}{8} < |x-z| \leq r_2 \atop |z|    > R+1} f(z)dz \right) \nonumber \\
& \leq &
\frac{1+C_8}{\eta} \left(h_1(X,r_1, r_2) + \frac{h_2(X,r_1)}{\inf_{|y|\leq r_2}\nu(y)}\right)
\|f\|_\infty \int_{|z| \leq R} \nu(x-z)dz \nonumber \\
& \qquad + &
\frac{1}{\eta} \, \left(h_1(X,r_1, r_2) \int_{|z-x| > r_2 \atop |z| > R+1}f(z)\nu(z-x)dz +
h_2(X, r_1) \int_{\frac{r_1}{8} < |x-z| \leq r_2 \atop |z|    > R+1} f(z)dz \right).
\end{eqnarray*}
Furthermore, by (A1.1)-(A1.2),
\begin{align} \label{eq:proc1}
f(x) \leq c \left\|f\right\|_{\infty} \nu(x)
+ \frac{1}{\eta} \, \left(h_1(X,r_1, r_2) \int_{|z-x| > r_2 \atop |z| > R+1 }f(z)
\nu(z-x)dz + h_2(X, r_1) \int_{\frac{r_1}{8} < |x-z| \leq r_2 \atop |z|    > R+1} f(z)dz \right),
\end{align}
where
$$
c:= \frac{C_5^2C_6^{\left\lceil R \right\rceil}(1+C_8)}{\eta} \left(h_1(X,r_1, r_2) +
h_2(X, r_1)\frac{1}{\inf_{|y| \leq r_2} \nu(y)}\right)|B(0,R)|.
$$
This gives
\begin{align} \label{eq:proc2}
f(x) & \leq  \|f\|_{\infty} \, (c \nu(x) + c_1), \quad |x| \geq R+1,
\end{align}
with
$$
c_1=\frac{1}{\eta}\left[h_1(X,r_1, r_2) \nu(B(0,r_2)^c) + h_2(X, r_1) |B(0,r_2)|\right].
$$
Also, write
$$
c_2=\frac{1}{\eta} \left(h_1(X,r_1, r_2) \, K^X_1(r_2)+
h_2(X, r_1) |B(0,r_2)| K^X_2(r_2, r_3, \infty)\right).
$$
Assumption \eqref{eq:cond1} and Lemma \ref{lem:tail_dom} imply that $c_1 \vee c_2 < 1$, which will be essential in what follows.

We now show that for every $p \in \N$
\begin{align} \label{eq:eq_ind}
f(x) \leq c \left\|f\right\|_{\infty} \sum_{i=1}^p c_2^{i-1} \nu(x) +
\left\|f\right\|_{\infty} c_1^p , \quad |x| \geq R+1.
\end{align}
Notice that if this holds, then by taking the limit $p \to \infty$ it follows that
$$
f(x) \leq \frac{c}{1-c_2}\left\|f\right\|_{\infty} \nu(x), \quad |x| \geq R+1,
$$
which is the bound stated in the lemma.

To prove \eqref{eq:eq_ind} we make induction on $p \in \N$. First observe that \eqref{eq:proc2} is just \eqref{eq:eq_ind}
for $p=1$. Suppose now that \eqref{eq:eq_ind} is true for $p-1 \in \N$. By \eqref{eq:proc1} and the induction hypothesis
we see for all $|x| \geq R+1$ that
\begin{align*}
f(x) & \leq c \left\|f\right\|_{\infty} \nu(x) + \frac{c\left\|f\right\|_{\infty}}{\eta}
\sum_{i=1}^{p-1} c_2^{i-1} \Bigg(h_1(X,r_1, r_2) \left\|f\right\|_{\infty}
\int_{|z-x|    > r_2 \atop |z| > R+1 }\nu(z) \nu(z-x)dz \\
& \ \ \ \ \ \ \ \ \ \ \ \ \ \ \ \ \ \ \ \ \ \ \ \ \ \ \ \ \ \ \ \ \ \ \ \ \ \ \ \ \ \ \ \ \ \ \ \ \ \ \ \ \ \ \ \
\qquad + h_2(X, r_1) |B(0,r_2)| \sup_{|z-x| \leq r_2} \nu(z) \Bigg) \\
& \quad + \frac{\left\|f\right\|_{\infty} c_1^{p-1}}{\eta}  \, \left(h_1(X,r_1, r_2)
\nu(B(0,r_2)^c) + h_2(X, r_1) |B(0,r_2)| \right).
\end{align*}
Applying \eqref{def:K1} to the integral in the second summand and using the definition of the
constants $c_1, c_2$ gives
\begin{align*}
f(x) & \leq c \left\|f\right\|_{\infty} \nu(x) \\
& \quad
+ c\left\|f\right\|_{\infty}  \sum_{i=1}^{p-1} \frac{c_2^{i-1}}{\eta} \left[h_1(X,r_1, r_2)
K^X_1(r_2) + h_2(X, r_1) |B(0,r_2)| K^X_2(r_2, r_3, \infty) \right] \nu(x) \\
& \quad +
\left\|f\right\|_{\infty} c_1^{p-1}  \, \frac{h_1(X,r_1, r_2) \nu(B(0,r_2)^c)+
h_2(X, r_1) |B(0,r_2)| }{\eta} \\
& \leq
c\left\|f\right\|_{\infty} \sum_{i=1}^{p} c_2^{i-1} \nu(x) + \left\|f\right\|_{\infty} c_1^{p},
\end{align*}
which completes the proof.
\end{proof}

The power of Lemma \ref{lem:defic1} depends on the verifiability of condition (\ref{eq:cond1}). As it will turn out, whether the L\'evy intensity has a slow or quick decay will make a difference, and the latter case is more difficult.

First we apply Lemma \ref{lem:defic1} to intensities which are slowly decaying in the sense of the following assumption:
\begin{align} \label{eq:nu_dbl}
  \begin{array}{c}
	 \mbox{there exists \ $C_{15}> 0$ \ such that for all \ $r \geq 1$ we have}  \\
   \mbox{$\nu(x-y) \leq C_{15} \nu(x)$ \ whenever \ $|y| \leq r$ \ and \ $|x| \geq 2r$.}
  \end{array}
\end{align}
\smallskip
\noindent
This property typically holds for functions polynomially decaying at infinity.
It is straightforward to check that under \eqref{eq:nu_dbl} we
have
\begin{align} \label{eq:bdd_by_tail}
K_1^X(s) \leq C_5^2 C_{15} \nu(B(0,s)^c), \quad \text{for any} \ \ s \geq 1,
\end{align}
and
$$
K_2^X(s_1, s_2, s_3) \leq C_{15}, \quad \text{for every} \ \ s_1 \geq 1, \ s_2 \geq 2s_1 \ \text{and} \ s_2 < s_3 \leq \infty.
$$
This immediately gives that under assumption \eqref{eq:nu_dbl} conditions \eqref{eq:intr_killing}-\eqref{eq:unif_bdd} hold
automatically with $\kappa_1 = 2$ and $\kappa_2=C_{15}$. By this fact, in the following theorem our standard set of assumptions
\eqref{eq:intr_killing}, \eqref{eq:unif_bdd}, \eqref{eq:Green_bdd} simplifies to \eqref{eq:nu_dbl} and \eqref{eq:Green_bdd}.

\begin{theorem}
\label{thm:defic2}
Let $\pro X$ be a L\'evy process with L\'evy-Khinchin exponent $\psi$ given by \eqref{eq:Lchexp}-\eqref{eq:nuinf} such that
Assumptions (A1)--(A3) hold. Moreover, suppose that conditions \eqref{eq:nu_dbl} and \eqref{eq:Green_bdd} hold. Then for every
$\eta   >0$ and every bounded non-negative function $f$ which is $(X,\eta)$-harmonic in $\overline B(0,r)^c$ for some $r > 0$,
there exist $C_{16} =C_{16}(X,\eta)$ and $R=R(X,\eta)$ such that
$$
f(x) \leq C_{16} \left\|f\right\|_{\infty} \nu(x), \quad |x| \geq R.
$$
\end{theorem}

\begin{proof}
Notice that whenever $s \geq 4$, under \eqref{eq:nu_dbl} we have $K^X_2(s/4, s/2, s) \leq C_{15}$, $K^X_2(s, 2s, \infty)
\leq C_{15}$, and $K^X_2(2s, 4s, \infty) \leq C_{15}$. Thus by \eqref{def:C9}, \eqref{eq:gen_est_tau} and \eqref{eq:Green_bdd}
we also have $h_1(X,s,2s) \leq c_1$ for the same $s$. This and \eqref{eq:bdd_by_tail} jointly imply that
$$
h_1(X,s,2s) K^X_1(2s) \leq c_1  K^X_1(2s) \to 0 \quad \text{as} \quad s \to \infty.
$$
Similarly, by \eqref{def:C9}, \eqref{eq:gen_est_tau}, \eqref{eq:Green_bdd} and the fact that $g(s)s^d \to 0$ as $s \to \infty$,
we get
$$
s^d \, h_2(X, s) \, K^X_2(2s, 4s, \infty) \to 0 \quad \text{as} \quad s \to \infty.
$$
To complete the proof, observe that for every $\eta>0$ there exists $s \geq 4$ such that condition \eqref{eq:cond1}
holds with $r_1=s$, $r_2=2r_1=2s$ and $r_3 = 4s > r_2$. The claimed bound follows from Lemma \ref{lem:defic1}.
\end{proof}

\subsection{Upper bound for the Laplace transform of the first hitting time of a ball}
\noindent
A specific harmonic function to which we need to apply the results of the previous subsection in
order to study the decay of eigenfunctions is of the type
\begin{align} \label{eq:modelhf}
f(x) =  \left\{
\begin{array}{lrl}
\ex^x\left[e^{-\eta \tau_{\overline B(0,r)^c}} \right] & \mbox{  for  } & |x| > r,
\vspace{0.2cm} \\
\1_{\{|x| \leq r\}}(x)  & \mbox{  for  } & |x| \leq r.
\end{array}\right.
\end{align}
Here we use the standard convention that $1/\infty = 0$.

Condition \eqref{eq:nu_dbl} fails to hold for L\'evy measures which are lighter than polynomial at infinity.
Although Theorem \ref{thm:defic2} cannot be directly extended to this class of processes, the structure involving the functions $h_1, h_2$ in \eqref{eq:harmest} allows to use Lemma \ref{lem:defic1} also in this case. However, for such L\'evy measures the function $K_2^X(s_1/4,s_1/2,s_1)$ appearing in $C_{13}(X,s_1)$ typically diverges faster than polynomially for large $s$, and $s_1^d \, h_2(X, s_1)$ does not vanish as $s_1 \to \infty$. This problem occurs when, for instance, $\nu(x)$ decays like $e^{-|x|^{\beta}}$, for some $\beta \in (0,1)$. This difficulty persists even when taking a fixed positive number $s_0$ instead of proportionally increasing the radius $s_1/4$ to replace $K_2^X(s_1/4,s_1/2,s_1)$ by $K_2^X(s_0,s_1/2,s_1)$. (Indeed, for L\'evy densities satisfying (A1) it is always true that we can choose a $C=C(s_0)$ such that $K_2^X(s_0,s_1/2,s_1) \leq C$ for large $s_1 > 0$, however, we see that in this case $C_3(X,s_1/16)$ in $h_2(X,s_1)$ must be replaced by a strictly positive constant independent of $s_1$ and again the expression $s_1^d \, h_2(X, s_1)$ cannot vanish as $s_1 \to \infty$.) Therefore, in Theorem \ref{thm:defic3} below giving the upper bound of \eqref{eq:modelhf} for L\'evy densities which are light at infinity we choose the strategy to keep all the radii in \eqref{eq:harmest} proportional to $s_1$. To get this result, we will use an argument based on the domination of $\nu$ by a carefully constructed family of jump intensities. Specifically, for any small $\eta >0$ we will choose a sufficiently large radius of harmonicity $s_1$ and construct a L\'evy process with jump intensity $\nu^{s_1}$ to which we can effectively apply Lemma \ref{lem:defic1} and which has the following properties:
\begin{itemize}
\item[(1)]
$\nu^{s_1}$ dominates $\nu$ inside $B(0,s_1)$ and agrees with $\nu$ outside $B(0,s_2)$ for sufficiently
large $s_2 \gg s_1$ proportional to $s_1$, but the difference $\nu^{s_1}- \nu$ is relatively small with respect to $\eta$.
\item[(2)]
the function \eqref{eq:modelhf} corresponding to $\nu$ is dominated by that of $\nu^{s_1}$ with $\eta/2$ at infinity.
\end{itemize}

To achieve this, we make use of the following comparison scheme.
\begin{proposition}
\label{prop:domination}
Let $(X^{(1)}_t)_{t \geq 0}$ and $(X^{(2)}_t)_{t \geq 0}$ be two L\'evy processes with characteristic exponents
$\psi^{(1)}, \psi^{(2)}$ as in \eqref{eq:Lchexp}, the same diffusion coefficient $A$, and L\'evy measures
$\nu^{(1)}, \nu^{(2)}$ satisfying \eqref{eq:nuinf} and such that $\sigma(dx)=\sigma(x)dx$ with 
$\sigma(x):= \nu^{(2)}(x) - \nu^{(1)}(x)$ for $x \neq 0$, and $\sigma(0):=0$, is a non-negative
finite measure. Denote $|\sigma|:= \sigma(\R^d)$. Then the following hold.
\begin{itemize}
\item[(1)]
For every $t>0$ and almost every $x \in \R^d$, we have
$$
e^{-|\sigma| t} p^{(1)}(t,x) \leq p^{(2)}(t,x) \leq e^{-|\sigma| t} p^{(1)}(t,x) + t
\, \sup_{z \in \R^d} \sigma(z), 
$$
where $p^{(1)}(t,x), \, p^{(2)}(t,x)$ are the transition densities of the processes $(X^{(1)}_t)_{t \geq 0}, (X^{(2)}_t)_{t \geq 0}$,
respectively.
\item[(2)]
For every $\eta > |\sigma|$ and almost every $x \in \R^d$, we have $G^{\eta}_1(x) \leq G^{\eta-|\sigma|}_2(x)$, where $G^{\eta}_{i}(x) = \int_0^{\infty} e^{-\eta t} p^{(i)}(t,x) dt$, $i=1,2$, are the $\eta$-potential kernels of the two processes, respectively.
\item[(3)]
If there exist $C_{17} \geq 1$ and $R > 0$ such that $p^{(1)}(t,x) \leq C_{17} p^{(1)}(t,y)$ for every $|x| \geq |y| \geq R$,
$|x-y| \leq 1$, and $t>0$, then for every $\eta > |\sigma|$ and $r>0$ there exists a constant $C_{18}=C_{18}(X^{(1)}, X^{(2)},\eta,r)$
such that for every $|x| > 3r + R$ we have
\begin{align} \label{eq:comp_Lapl}
\ex^x[e^{-\eta \tau^{(1)}_r}] \leq C_{17}^{2 \left\lceil 2r \right\rceil} C_{18} \ex^{x-x_{2r}}[e^{-(\eta - |\sigma|)
\tau^{(2)}_{r}}],
\end{align}
where $x_r:=(r/|x|) x$ and $\tau^{i}_r:= \inf\left\{t>0: X^{(i)}_t \in \overline{ B}(0,r)
\right\}$, $i=1,2$.
\end{itemize}
\end{proposition}

\begin{proof}
First consider (1).
Recall that the compound Poisson measure corresponding to $\sigma$ is given by
$e^{t \sigma} = e^{-t |\sigma|} \sum_{n=0}^{\infty} \frac{t^n \sigma^{n*}}{n!}$, $t >0$. Since $\nu^{(2)} =
\nu^{(1)} + \sigma$, we can write
$$
p^{(2)}(t,x) = p^{(1)}(t,\cdot) * \exp( t \sigma)(x) = p^{(1)}(t,x) e^{- t |\sigma|} + e^{-t|\sigma|}
\sum_{n=1}^{\infty} \frac{t^n p^{(1)}(t,\cdot) * \sigma^{n*}(x)}{n!}, \ \ t>0, \, \text{a.e.} \, x \in \R^d.
$$
From this we see that the first inequality in (1) holds. To show the second, it suffices
to estimate the sum at the right hand side. We have
$$
p^{(1)}(t,\cdot) * \sigma^{n*}(x) =  \int_{\R^d} \sigma^{n*}(x-y) p(t,y) dy \leq \sup_{z \in \R^d} \sigma^{n*}(z)
\leq |\sigma|^{n-1} \sup_{z \in \R^d} \sigma(z),
$$
for every $t>0$, $n \in \N$ and almost every $x \in \R^d$. Hence,
$$
p^{(2)}(t,x) \leq e^{- t |\sigma|}p^{(1)}(t,x)  +  t  \sup_{z \in \R^d} \sigma(z) \, e^{-t|\sigma|}
\sum_{n=1}^{\infty} \frac{t^{n-1} |\sigma|^{n-1}}{n!} \leq e^{- t |\sigma|}p^{(1)}(t,x)  + t  \sup_{z \in \R^d} \sigma(z),
$$
for all $t>0$ and almost every $x \in \R^d$, which proves the claim.
Both inequalities in (1) directly extend to the case when $A$ is non-zero. Indeed, then for every fixed $t>0$ the resulting
transition density is a convolution of the Gaussian measure and the kernel corresponding to the jump part of the process,
for which the required bounds are proven above.

Assertion (2) is a direct consequence of the first inequality in (1), since for every $\eta>0$ and almost every $x \in \R^d$, we have
$G^{\eta}_{i}(x) = \int_0^{\infty} e^{-\eta t} p^{(i)}(t,x) dt$, $i=1,2$.

Now we show (3). Denote by $\mu^{\eta, (i)}_{B(x_0,r)}(dx)$ the $\eta$-capacitory measure of the ball
$B(x_0,r)$, $x_0 \in \R^d$, for the process $(X^{(i)}_t)_{t \geq 0}$ (see e.g. \cite[Sect. 2, Ch. II]{bib:Ber}).
It is known to be a Radon measure with support contained in $\overline{B}(x_0,r)$. By \cite[Th. 7 and pp.51-52]{bib:Ber}, for every $\eta>0$, $x \in \R^d$ and $r>0$ we
have
\begin{align} \label{eq:capacitory}
\ex^x[e^{-\eta \tau^{(i)}_{\overline B(x_0,r)^c}}] = \int_{\overline{B}(x_0,r)} G^{\eta}_i(x-y) \mu^{\eta, (i)}_{B(x_0,r)}(dy),
\quad i=1,2.
\end{align}
When $x_0=0$, we write $\tau_r^{(i)}$ and $\mu^{\eta, (i)}_r(dx)$ for a shorthand. From here, for every $\eta > |\sigma|$, $r>0$
and $|x| > 3r+R$ we obtain
$$
\ex^x[e^{-\eta \tau^{(1)}_r}] \leq \int_{\overline{B}(0,r)} G^{\eta}_1(x-y)
\mu^{\eta, (1)}_r(dy) \leq \sup_{y \in \overline B(0,r)} G^{\eta}_1(x-y) \ \mu^{\eta, (1)}_r (\overline{B}(0,r)).
$$
By the local uniform comparability of the densities $p^{(1)}(t,x)$, for every $|x|>3r+R$ we have
\begin{align*}
\sup_{y \in \overline B(0,r)} G^{\eta}_1(x-y)
& \leq
\int_0^{\infty} e^{-\eta t} \sup_{y \in \overline B(0,r)} p^{(1)}(t,x-y) dt
\leq
C_{17}^{\left\lceil 2r \right\rceil} \int_0^{\infty} e^{-\eta t} p^{(1)}(t,x-x_r) dt \\
& =
C_{17}^{2 \left\lceil 2r \right\rceil} \int_0^{\infty} e^{-\eta t} \inf_{y \in \overline B(x_{2r},r)} p^{(1)}(t,x-y) dt
\leq
C_{17}^{2 \left\lceil 2r \right\rceil} \inf_{y \in \overline B(x_{2r},r)} G^{\eta}_1(x-y).
\end{align*}
Moreover, $\mu^{\eta-|\sigma|, (2)}_r(\overline{B}(0,r)) = \mu^{\eta-|\sigma|, (2)}_{B(x_{2r},r)}
(\overline{B}(x_{2r},r))$. Putting together the estimates above and using (2), we finally obtain
\begin{align*}
\ex^x[ & e^{-\eta \tau^{(1)}_r}]
\leq
\mu^{\eta, (1)}_r(\overline{B}(0,r)) \sup_{y \in \overline B(0,r)} G^{\eta}_1(x-y) \\
& \leq
C_{17}^{2 \left\lceil 2r \right\rceil} \frac{\mu^{\eta, (1)}_r(\overline{B}(0,r))}{\mu^{\eta-|\sigma|, (2)}_r
(\overline{B}(0,r))}\mu^{\eta-|\sigma|, (2)}_{B(x_{2r},r)}(\overline{B}(x_{2r},r))
\inf_{y \in \overline B(x_{2r},r)} G^{\eta}_1(x-y) \\
& \leq
C_{17}^{2 \left\lceil 2r \right\rceil} \frac{\mu^{\eta, (1)}_r(\overline{B}(0,r))}{\mu^{\eta-|\sigma|, (2)}_r
(\overline{B}(0,r))} \int_{\overline B(x_{2r},r)} G^{\eta-|\sigma|}_2(x-y) \mu^{\eta-|\sigma|, (2)}_{B(x_{2r},r)}(dy)
  \\
& \leq
C_{17}^{2 \left\lceil 2r \right\rceil} \frac{\mu^{\eta, (1)}_r(\overline{B}(0,r))}{\mu^{\eta-|\sigma|, (2)}_r
(\overline{B}(0,r))} \ex^x[e^{-(\eta-|\sigma|) \tau^{(2)}_{\overline{B}(x_{2r},r)}}]
=
C_{17}^{2 \left\lceil 2r \right\rceil} \frac{\mu^{\eta, (1)}_r(\overline{B}(0,r))}{\mu^{\eta-|\sigma|, (2)}_r
(\overline{B}(0,r))} \ex^{x-x_{2r}}[e^{-(\eta-|\sigma|) \tau^{(2)}_r}],
\end{align*}
where $\tau^{(2)}_{\overline{B}(x_{2r},r)^c} = \inf\left\{t>0: X^{(2)}_t \in \overline{ B}(x_{2r},r) \right\}$.
\end{proof}

Making use of the above results, we can now derive an upper bound for the function defined in
\eqref{eq:modelhf} for symmetric jump-paring processes with L\'evy measures which are light at infinity in the sense that
\begin{align} \label{eq:fsm}
\int_{\R^d} |x|^2 \nu(dx) < \infty.
\end{align}
Since in this case property \eqref{eq:nu_dbl} in general does not hold, we require instead that
\begin{align} \label{eq:comp_dens}
  \begin{array}{c}
	 \mbox{there exist $C_{17} \geq 1$ and $R   >0$ such that $p(t,x) \leq C_{17} p(t,y)$}  \\
	 \mbox{for every $t>0$ and $|x| \geq |y| \geq R$ satisfying $|x-y| \leq 1$.}
	\end{array}
\end{align}

\vspace{0.1cm}

\noindent
This means that the transition probability densities in some sense inherit the properties (A1.1)--(A1.2) of the
L\'evy kernel, which is a reasonable requirement on the processes with jump intensities described by (A1). Note
also that due to condition \eqref{eq:fsm} we do not impose \eqref{eq:Green_bdd}, and only assume that the general estimate \eqref{eq:suf1} holds. Unlike the pivotal assumptions \eqref{eq:intr_killing}--\eqref{eq:unif_bdd}, both \eqref{eq:comp_dens} and \eqref{eq:suf1} should be seen as only technical assumptions providing a framework to our study. Note that they can be verified efficiently for a substantial subclass of jump-paring L\'evy processes. For instance, if $A \equiv a\Id$ for some $a \geq 0$ and $\nu$ is a non-increasing radial function (i.e., $\pro X$ is an isotropic unimodal L\'evy process), then \eqref{eq:comp_dens} and \eqref{eq:suf1} automatically hold \cite{bib:BGR}. For more general processes with jump intensities satisfying (A1) which are not non-increasing radial functions, they can be verified under a mild regularity assumption on the profile $g$ around zero (see Proposition \ref{prop:bound_by_psi} below).

\begin{theorem}
\label{thm:defic3}
Let $\pro X$ be a L\'evy process with L\'evy-Khinchin exponent $\psi$ given by \eqref{eq:Lchexp}--\eqref{eq:nuinf} such
that Assumptions (A1)--(A3) are satisfied. Moreover, suppose that conditions \eqref{eq:suf1},
\eqref{eq:fsm}--\eqref{eq:comp_dens} and \eqref{eq:intr_killing}--\eqref{eq:unif_bdd} hold. Then for every $\eta>0$ and
$r>0$ there exist $C_{19}=C_{19}(X,\eta,r)$ and $R=R(X,\eta,r) > r$ such that
$$
\ex^x \left[e^{-\eta \tau_{\overline B(0,r)^c} }\right] \leq C_{19} \, \nu(x), \quad |x| \geq R.
$$
\end{theorem}

\begin{proof}
We define a family of L\'evy measures $(\nu^{s})_{s \geq 4}$ by $\nu^{s}(dx)=\nu^{s}(x)dx$, with
\begin{align} \label{def:defnu}
\nu^{s}(x) := \left\{
  \begin{array}{ll}
	  \sup_{s/4 \leq |y| \leq s} \nu(y) & \mbox{  for    } s/4 \leq |x| \leq s, \vspace{0.2cm} \\
    \nu(x)  & \mbox{  otherwise. }
  \end{array}\right.
\end{align}
Let $\pro{X^{s}}$ be a L\'evy process determined by the L\'evy-Khinchin exponent $\psi^{s}$ (defined as in
\eqref{eq:Lchexp}, with L\'evy triplet $(0,A,\nu^{s}))$. We also define the corresponding symmetrization
$\Psi^{s}$ and profile function $H^{s}$ as in \eqref{eq:Lchexpprof} and \eqref{eq:PruitH}, respectively.
Note that $\nu^{s} = \nu + \sigma$, where $\sigma$ is a finite non-negative measure with density
$\sigma(x) = (\sup_{s/4 \leq |y| \leq s} \nu(y) - \nu(x))\1_{\{s/4 \leq |x| \leq s\}}$. Moreover, 
$$\int_{|x|/2 < |y| < |x|} |y|^2\nu(y) dy \geq (4 C_5)^{-1} g(|x|)|x|^2 \, |\left\{y: |x|/2<|y|<|x|\right\}|,$$ 
for $|x| \geq 1$. Thus, by \eqref{eq:fsm}, we have $\nu(x) \leq \nu^{s}(x) \wedge c_1 |x|^{-d-2}$, for all $|x| \geq 1$ and $s \geq 4$.
This and \eqref{eq:PruitH} immediately give
\begin{align} \label{eq:est0}
c_2 r^2 \leq \Psi(r) \leq \Psi^{s}(r) \leq c_3 H^{s} \left(\frac{1}{r}\right)
\leq c_4 (r^2 + s^{-2}), \quad r \in (0,1], \ \ \ s \geq 4,
\end{align}
with constants $c_2, c_3, c_4$ independent of $s$ and $r$. For more clarity, we divide the remainder
of the proof into four steps.
\smallskip

\noindent
\textit{Step 1.} In this step we estimate the function $K^{X^s}_3(r)$, the constants $C_3(X^s,r)$ given by
\eqref{def:C9a}, and the mean exit time from a ball for the process $\pro{X^{s}}$ with L\'evy measure $\nu^{s}$,
$s \geq 4$. As a consequence we also obtain general upper bounds for the corresponding functions $h_1, h_2$.

First notice that by \eqref{def:C9}, \eqref{eq:gen_est_tau} and \eqref{eq:est0}, we have
\begin{align} \label{eq:est1}
\ex^0[\tau^{X^{s}}_{B(0,r)}] \leq c_5 r^2 \quad \text{and}
\quad C_3(X^{s}, r) \leq c_6 (r^{-2} + s^{-2}),
\end{align}
for every $r \geq 1$ and $s \geq 4$, with constants $c_5, c_6$ independent of $s$ and $r$. Moreover,
we can also prove that there exists a constant $c_7$ (also uniform in $s$ and $r$) such that
\begin{align} \label{eq:est2}
K^{X^s}_3(r) \leq c_7 r^{4-d} (r^{-2} + r^d s^{-d-2}), \quad r \geq 8, \ \ \ s \geq 4.
\end{align}
Indeed, by the second inequality in Proposition \ref{prop:domination}(1), the transition densities
$p^{s}$ of $\pro{X^{s}}$ satisfy $p^s(t,x) \leq p(t,x) + t \sup_{z \in \R^d} \sigma(z)$, $x \in \R^d$,
$t>0$. Under assumption \eqref{eq:suf1}, this implies that $\sup_{|x| \geq r} p^s(t,x) \leq c_8 t
(\Psi(1/r)r^{-d} + s^{-d-2})$, for every $t > 0$, $r \geq 1$ and $s \geq 4$, with the constant $c_8$
independent of $s, r, t$. Moreover, by \cite[Lem. 5(a)]{bib:KS14}, (A2) and \eqref{eq:est0}, we obtain for $t>t_b$
\begin{align*}
\int_{\R^d} e^{-t \psi^s(\xi)} d\xi & \leq e^{-(t-t_b) \inf_{|z| \geq 1} \psi(z)} \int_{|\xi| \geq 1}
e^{- t_b \psi(\xi)} d\xi + \int_{|\xi| < 1} e^{-c_9 t |\xi|^2} d\xi \leq \frac{c_{10}}{t^{d/2}} \leq
c_{11} \left(\Psi^{-1}\left(\frac{1}{t}\right)\right)^d.
\end{align*}
Therefore, \eqref{eq:est2} follows from Lemma \ref{lem:ver_cond4} applied to $\pro {X^s}$, with $\Phi=
\Psi$ and $\Theta=s^{-d-2}$.

Next we take $\pro{X^{s}}$ with L\'evy measure $\nu^{s}$, $s \geq 4$, and consider the functions $h_1, h_2$. Using the bounds
\eqref{eq:est1}-\eqref{eq:est2}, for every $r_1 \geq 8$, $r_2 \geq 2r_1$ and $s \geq 4$ we
obtain
\begin{align} \label{eq:est4}
h_1(X^{s},r_1, r_2) \leq K^{X^s}_2(r_1, r_2, \infty)
\Bigg[c_{12} \left(1+\left(\frac{r_1}{s}\right)^2 \right)
\Bigg(1+\left(\frac{r_1}{s}\right)^{d+2}   +
\left(K^{X^{s}}_2\left(\frac{r_1}{4},\frac{r_1}{2},r_1\right)\right)^2 \Bigg) + 1 \Bigg]
\end{align}
and
\begin{eqnarray} \label{eq:est5}
r_1^d \, h_2(X^{s}, r_1)
& \leq &
\frac{c_{13}}{r_1^2} \left(1+\left(\frac{r_1}{s}\right)^2 \right) \Bigg[\left(1+\left(\frac{r_1}{s}\right)^2 \right)
\Bigg(1+\left(\frac{r_1}{s}\right)^{d+2}   +
\left(K^{X^{s}}_2\left(\frac{r_1}{4},\frac{r_1}{2},r_1\right)\right)^2 \Bigg) \nonumber \\
&& \qquad +
\left(1 \vee \left(\frac{r_1}{s} \right) \right)^{d+2} \Bigg] + c_{14} r_1^d \left(\frac{1}{r_1^{d+2}} \vee
\frac{1}{s^{d+2}} \right),
\end{eqnarray}
where the constants $c_{12}, c_{13} ,c_{14}$ do not depend on $s$, $r_1$ and $r_2$. The function $K^{X^{s}}_2$
appearing in the above bounds still does depend on the specific form of $\nu^{s}$, i.e., on $s$.
\smallskip

\noindent
\textit{Step 2.}
Now we choose a specific $s$ to obtain a suitable form of the function $K^{X^{s}}_2$ using the general upper bounds for the functions $h_1$, $h_2$ corresponding to $\pro{X^{s}}$ with L\'evy measure $\nu^{s}$ in the previous step. Note that while in the previous step the constants $c_{12}, c_{13}, c_{14}$ are independent of $s$, and $r_1, r_2$ are such that $s \geq 4$, $r_1 \geq 8$ and $r_2 \geq 2r_1$, from now on we assume that $s = r_1$ in \eqref{def:defnu}.

By the definition of $\nu^{r_1}$ we see that $\nu^{r_1}(x) = \nu(x)$ for $|x|  > r_1$ and $\nu^{r_1}(x) =
\sup_{r_1/4 \leq |z| \leq r_1}\nu(z)$ for $r_1/4 \leq |x| \leq r_1$. With this, we obtain for $r_1 \geq 8$ and $r_2 \geq 2r_1$ that
$$
K^{X^{r_1}}_1(r_2) = K^{X}_1(r_2), \quad K^{X^{r_1}}_2(r_1, r_2, \infty) = K^X_2(r_1, r_2, \infty) \quad \mbox{and} \quad K^{X^{r_1}}_2
\left(\frac{r_1}{4},\frac{r_1}{2},r_1\right) \leq C_5^2.
$$
The latter bound is a consequence of the fact that for $|y| \leq r_1/4$ and $r_1/2 \leq |x| < r_1$ we have
$$
\nu^{r_1}(x-y) = \nu^{r_1}(x), \quad \frac{r_1}{4} \leq |x-y| \leq r_1
$$
and
$$
\nu^{r_1}(x-y) = \nu(x-y) \leq C_5 g(|x-y|) \leq C_5 g(|x|) \leq C_5^2 \nu(x) \leq C_5^2 \nu^{r_1}(x), \quad r_1 < |x-y| \leq \frac{5r_1}{4}.
$$
In particular, by \eqref{eq:est4}-\eqref{eq:est5} and the discussion of the function $K^{X^{r_1}}_2$ above, we arrive at
$$
h_1(X^{r_1}, r_1, r_2) \leq c_{15} K^X_2(r_1,r_2, \infty), \quad r_1 \geq 8, \ \ \ r_2 \geq 2r_1,
$$
and
$$
r_1^d h_2(X^{r_1}, r_1) \leq \frac{c_{16}}{r_1^2}, \quad r_1 \geq 8.
$$
Note also that for any $r_2 \geq 2r_1$ and $r_3  > r_2 + r_1$ we have $K^{X^{r_1}}_2(r_2, r_3, \infty) =
K^X_2(r_2, r_3, \infty)$.

\smallskip

\noindent
\textit{Step 3.} Let now $\kappa_1 \geq 2$ and $\kappa_2 < \infty$ be the parameters given by
\eqref{eq:intr_killing}-\eqref{eq:unif_bdd}. From the construction made in the previous two steps we obtain that for every
$\eta    > 0$ there is $r_0 =r_0(\eta)$ such that for every $r_1 \geq r_0$ there exists a L\'evy process $\pro{X^{r_1}}$ with
L\'evy measure $\nu^{r_1}$  given by \eqref{def:defnu}, $r_2 = \kappa_1 r_1 \geq 2r_1$ and $r_3 > r_2+r_1$ for which
$$
2d \, C_5^4 \,h_1(X^{r_1},r_1, r_2) \, K^{X^{r_1}}_1(r_2) + h_2(X^{r_1}, r_1)
|B(0, r_2)| K^{X^{r_1}}_2(r_2, r_3, \infty) < \eta.
$$
We thus proved that for given $\eta   >0$ there exists $r_0=r_0(\eta)$ such that for every L\'evy process
$\pro{X^{r_1}}$ with $r_1 \geq r_0$ the assumption \eqref{eq:cond1} of Lemma \ref{lem:defic1} is satisfied
with $r_1$, $r_2$ and $r_3$. Hence for every $\eta>0$ there exists
$r_0=r_0(\eta)$ such that for all $r_1 \geq r_0$ and $R_1 > 0$ there exists a constant
$c_{17}$ for which
\begin{align} \label{eq:concl3}
\ex^x \left[e^{-\eta \tau^{X^{r_1}}_{\overline B(0,R_1)^c}} \right] \leq c_{17} \nu^{r_1}(x), \quad |x| \geq R_2.
\end{align}
where $\tau^{X^{r_1}}_{\overline B(0,R_1)^c} = \inf\left\{t>0: X^{r_1}_t \in \overline{ B}(0,R_1) \right\}$
and $R_2 = (r_1+R_1) \vee r_3$.

\smallskip 

\noindent
\textit{Step 4.} We now complete the proof of the claimed bound for the initial L\'evy process with L\'evy measure
$\nu$. Let $\eta >0$ and $r_0=r_0(\eta/2)$ be such that for every $r_1 \geq r_0$ the estimate \eqref{eq:concl3}
holds for the process $\pro{X^{r_1}}$ with L\'evy measure $\nu^{r_1}$ and $\eta/2$ instead of $\eta$. Choose $r_1=
r_1(\eta) \geq r_0$ such that $|B(0,r_1)| \sup_{r_1/4 \leq |z| \leq r_1}\nu(z) \leq c_{18} r_1^{-2} \leq \eta/2$
and recall that we have $\nu(x) \leq \nu^{r_1}(x)$, $x \in \R^d \backslash \left\{0\right\}$. 
This means that the assumptions of Proposition \ref{prop:domination}
are satisfied with $\nu^{(1)} = \nu$, $\nu^{(2)} = \nu^{r_1}$ and the corresponding $\sigma$, and
$|\sigma| \leq \eta/2$. Thus by \eqref{eq:comp_Lapl} and \eqref{eq:concl3} we finally obtain
$$
\ex^x \left[e^{-\eta \tau^{X^{r_1}}_{\overline B(0,R_1)^c}} \right] \leq c_{19} \ex^{x-x_{2R_1}}
\left[e^{-(\eta/2)\tau^{X^{r_1}}_{\overline B(0,R_1)^c}} \right] \leq c_{20} \nu^{r_1}(x-x_{2R_1}),
\quad |x| \geq (R_2 \vee R) + 3R_1,
$$
where $R$ comes from \eqref{eq:comp_dens}. The conclusion follows now from the fact that
$\nu^{r_1}(x) = \nu(x)$ for $|x| > r_1$ and (A1.1)-(A1.2).
\end{proof}

Recall that if $\pro X$ is an isotropic unimodal L\'evy process, then both \eqref{eq:comp_dens} and
\eqref{eq:suf1} hold. We conclude this section by a general sufficient condition on the profile $g$
around zero, which allows to
extend this property to jump intensities as in (A1).

\begin{proposition} \label{prop:bound_by_psi}
Let $\pro X$ be a L\'evy process determined by the L\'evy-Khinchin exponent $\psi$ as in \eqref{eq:Lchexp}--\eqref{eq:nuinf}
with L\'evy triplet $(0,A,\nu)$, such that $A = a \Id$ for some $a \geq 0$ and $\nu$ obeys (A1). Moreover, suppose that
\eqref{eq:fsm} is satisfied and there exist $\gamma_1, \gamma_2 \in (0,2)$ and $C_{20}, C_{21}$ such that
\begin{align} \label{eq:weakscaling}
C_{20} \lambda^{-d-\gamma_1} g(r) \leq g(\lambda r) \leq C_{21} \lambda^{-d-\gamma_2} g(r), \quad \lambda \in (0,1], \ \ \
r \in(0, 1].
\end{align}
Then conditions \eqref{eq:comp_dens} and \eqref{eq:suf1} are satisfied. In particular, $\pro X$ is a symmetric jump-paring
L\'evy process, i.e., both of the remaining Assumptions (A2)--(A3) in Definition \ref{jumpparing} hold.
\end{proposition}

\begin{proof}
Let $\psi_{\nu}(\xi) = \int_{\R^d \backslash \left\{0\right\}}(1-\cos(\xi \cdot z)) \nu(z) dz$, $\xi \in \R^d$. Denote the
corresponding symmetrization of the characteristic exponent and the transition densities by $\Psi_{\nu}$ and $p_{\nu}$,
respectively. Clearly, $\psi(\xi) = a|\xi|^2 + \psi_{\nu}(\xi)$. Also, we denote by $p_a(t,x,y)=p_a(t,y-x)=(4 \pi at)^{-d/2}
\exp(-|y-x|^2/(4at))$ the transition densities of the diffusion part of $\pro X$, whenever $a>0$. It follows from
\eqref{eq:PruitH} that $c_1 \lambda^{\gamma_1} \Psi_{\nu}(r) \leq \Psi_{\nu}(\lambda r) \leq c_2 \lambda^{\gamma_2}
\Psi_{\nu}(r)$, $\lambda, r \geq 1$. By this, \cite[Lem. 5 (b)]{bib:KS14} and (A1.3), we see that both conditions (1.1) (a)
and (b) in \cite[Th. 1]{bib:KS14} are satisfied. Thus there exist $c_3, \theta, t_0 >0$ such that for every $t \in (0,t_0]$
\begin{align} \label{eq:basic_estimate}
p_{\nu}(t,x) \asymp c_3 \left(h(t)^{-d} \1_{\left\{|x| \leq \theta h(t) \right\}} + t g(|x|) \1_{\left\{|x| \geq \theta h(t)
\right\}} \right), \quad \text{with} \quad h(t):= \frac{1}{\Psi^{-1}_{\nu}\left(\frac{1}{t}\right)}.
\end{align}
In fact, by \cite[Prop. 1]{bib:KS14} we may also assume that $p_{\nu}(t,x) \leq c_3 (h(t)^{-d} \wedge t g(|x|))$,
$t \in (0,t_0]$, $x \in \R^d$.

We are now in the position to prove \eqref{eq:comp_dens}. We need to consider only the case $a>0$; for $a=0$ the proof is
similar and simpler. Let first $t \in (0,t_0]$ and $|x| \geq |y|$. By \eqref{eq:basic_estimate}, we have
\begin{align*}
p(t,x) = p_{\nu}(t,\cdot) * p_a(t,\cdot)(x) \leq c_3 \int_{\R^d}(h(t)^{-d} \wedge t g(|x-z|))p_a(t,z)dz = c_3
\int_{\R^d}k(t,x-z) p_a(t,z)dz,
\end{align*}
where $k(t,x):= h(t)^{-d} \wedge t g(|x|)$. Observe that for every fixed $t \in (0,t_0]$ both $k(t,\cdot)$ and
$p_a(t, \cdot)$ are non-increasing radial functions, and the convolution of such functions preserves this property.
Thus $\int_{\R^d}k(t,x-z) p_a(t,z)dz \leq \int_{\R^d}k(t,y-z) p_a(t,z)dz$ and hence, $p(t,x) \leq c_3^2 p(t,y)$,
for $|x| \geq |y|$ and $t \in (0,t_0]$. Let now $t > t_0$ and $|x| \geq |y| \geq 1$, $|x-y| \leq h(t_0)/2$. By
\eqref{eq:basic_estimate} we have
\begin{align*}
p(t &,x)  = \int_{\R^d} \int_{\R^d} p_{\nu}(t_0,x-z-w) p_{\nu}(t-t_0,w) dw \, p_a(t,z)dz \\
       & \leq c_3 \int_{\R^d} \int_{\R^d}(h(t_0)^{-d} \wedge t_0 g(|x-z-w|)) p_{\nu}(t-t_0,w) \, p_a(t,z) \, dw dz \\
       & \leq c_3 \left(\int \int_{|y-z-w| \leq \theta h(t_0)} h(t_0)^{-d}
			 +\int \int_{|y-z-w| \geq \theta h(t_0)} t_0 g(|x-z-w|) \right) p_{\nu}(t-t_0,w) p_a(t,z) \, dw dz.
\end{align*}
Notice that by (A1.1)--(A1.2) there exists a constant $c_4=c_4(t_0) \geq 1$ for which $g(|x-z-w|) \leq c_4
g(|y-z-w|)$ on the set $|y-z-w| \geq \theta h(t_0)$ and thus the sum of the above two integrals is bounded
by $c_3^2 c_4 p(t,y)$. Condition \eqref{eq:comp_dens} follows immediately.

We now prove \eqref{eq:suf1} for the density $p_{\nu}$ (i.e., when $a=0$). We already proved that $p_{\nu}(t,x)
\leq c_3 t g(|x|)$, $x \in \R^d$, $t \in (0,t_0]$. Since $g(r) r^d \leq c_5 \Psi_{\nu}(1/r)$, $r>0$, condition
\eqref{eq:suf1} follows from this for $t \in (0,t_0]$. Therefore it suffices to consider only the case $t > t_0$. By
\cite[Lem. 5 (a)]{bib:KS14}, \eqref{eq:fsm}, \eqref{eq:weakscaling} and \eqref{eq:PruitH}, we have $\psi_{\nu}(\xi) \geq
c_6 \Psi_{\nu}(|\xi|) \geq c_7 (|\xi|^2 \wedge |\xi|^{\gamma_1})$, $\xi \in \R^d$. With this, for $t>t_0$ we obtain
\begin{align} \label{eq:aux_est_1}
\int_{\R^d} e^{-t \psi_{\nu}(\xi)} |\xi| d\xi & = \int_{|\xi| \geq 1} e^{-t \psi_{\nu}(\xi)} |\xi| d\xi + \int_{|\xi| < 1}
e^{-t \psi_{\nu}(\xi)} |\xi| d\xi \nonumber \\       & \leq e^{-c_6(t-t_0) \Psi_{\nu}(1)} \int_{|\xi| \geq 1} e^{-c_7 t_0 |\xi|^{\gamma_1}}
|\xi| d\xi + \int_{|\xi| < 1} e^{-c_7 t |\xi|^2} |\xi| d\xi \\
& \leq c_8 \left(e^{-c_6t \Psi_{\nu}(1)} + t^{-(d+1)/2} \int_{\R^d} e^{-c_7 |\xi|^2} |\xi| d\xi  \right) \leq c_9 t^{-(d+1)/2}. \nonumber
\end{align}
Since $\Psi_{\nu}^{-1}(1/t) \asymp 1/\sqrt{t}$, $t    > t_0$, we see that the assumption (3) of \cite[Th. 1]{bib:KS13} is
satisfied for $T = (t_0, \infty)$. Moreover, by the monotonicity and the doubling property of $\Psi_{\nu}$, we can also
directly check that the remaining assumptions (1)--(2) of this theorem hold for $f(r):= \Psi_{\nu}(1/r)r^{-d}$ and
$\gamma = d$. Hence we get
$$
p_{\nu}(t,x) \leq c_{10} \left(\frac{t \, \Psi_{\nu}(1/|x|)}{|x|^{d}} + t^{-d/2} e^{-c_{11}\frac{|x|}{\sqrt{t}}
\log \left(1+c_{11}\frac{|x|}{\sqrt{t}}\right)}\right), \quad x \in \R^d \backslash \left\{0\right\}, \ \ \ t > t_0.
$$
It suffices to estimate the exponential term only. If $|x| \geq \sqrt{t}$, then
$$t
^{-d/2} e^{-c_{11}\frac{|x|}{\sqrt{t}} \log \left(1+c_{11}\frac{|x|}{\sqrt{t}}\right)}
\leq c_{12} t^{-d/2} (\sqrt{t}/|x|)^{d+2} = c_{12} t/|x|^{d+2},
$$
for a constant $c_{12} >0$. On the other hand, for $|x| \leq \sqrt{t}$ we similarly have
$$
t^{-d/2} e^{-c_{11}\frac{|x|}{\sqrt{t}} \log \left(1+c_{11}\frac{|x|}{\sqrt{t}}\right)}
\leq t^{-d/2} \leq t^{-d/2}(\sqrt{t}/|x|)^{d+2} =  t/|x|^{d+2},
$$
and thus finally obtain
$$
p_{\nu}(t,x) \leq c_{13} t \left(\frac{\Psi_{\nu}(1/|x|)}{|x|^{d}} + \frac{(1/|x|^2)}{|x|^{d}}\right), \quad x \in \R^d
\backslash \left\{0\right\}, \ \ \ t > t_0.
$$
Since by \eqref{eq:PruitH} we have $\Psi_{\nu}(r) \geq (c_7/c_6) r^2$, $r \in (0, 1]$, this implies \eqref{eq:suf1} for
$p_{\nu}$ when $t > t_0$. We thus completed the proof of \eqref{eq:suf1} for $p_{\nu}$ (i.e., when $a = 0$). Suppose now
that $a>0$. We have for every $r > 0$, $|x| \geq r$ and $t>0$, that
\begin{align*}
p(t,x) = \left(\int_{|y-x|\geq r/2} + \int_{|y-x| < r/2} \right) p_{\nu}(t,x-y) p_a(t,y) dy
\leq c_{14} t \left(\frac{\Psi_{\nu}(1/r)}{r^{d}} + \frac{(1/r^2)}{r^{d}}\right) + t^{-d/2} e^{-\frac{r^2}{16at}}.
\end{align*}
Similarly as above, we can show that $t^{-d/2} e^{-r^2/(16at)} \leq c_{15} t /r^{d+2}$. Since $\Psi(r) \asymp \Psi_{\nu}(r)
\asymp r^2$, $ r \in (0,1]$, we conclude that \eqref{eq:suf1} holds for $p$ as well.

Assumption (A2) follows since the integral $\int_{|\xi| \geq 1} e^{-c_7 t_0 |\xi|^{\gamma_1}} |\xi| d\xi$ is convergent, while
(A3) holds by the bound $\sup_{|x| \geq r} p(t,x) < c_{16} t$, $r   >0$, with $c_{16}=c_{16}(r)$, and the general estimate as
in \eqref{eq:est3}.
\end{proof}
\noindent
Note that in fact assumption \eqref{eq:suf1} can be directly extended to the case when $\xi \cdot A \xi \geq C |\xi|^2$,
$\xi \in \R^d$, for some $C>0$ (i.e., under the uniform ellipticity condition). The same should be true for \eqref{eq:comp_dens}.
We also conjecture that the above proposition holds in a greater generality and the requirement \eqref{eq:weakscaling} can be
relaxed.

\smallskip
\section{Spatial decay of eigenfunctions}
\label{sec:eig_decay}

\subsection{Decaying potentials and basic properties of eigenfunctions}
\noindent
Now we turn to discussing the spatial decay properties of eigenfunctions of non-local Schr\"odinger operators presented in
the Introduction. Except for the last subsection, in this part we consider \emph{decaying potentials} in the following sense:

\medskip

\begin{itemize}
\item[\textbf{(A4)}]
Let $V \in \cK_{\pm}^X$ be such that $V(x) \to 0$ as $|x| \to \infty$.
\end{itemize}

\medskip

The number $\lambda \in \R$ is an eigenvalue of the non-local Schr\"odinger operator $H$ if there exists an
eigenfunction $\varphi \in L^2(\R^d)$ such that
\begin{align} \label{eq:eig}
H \varphi = \lambda \varphi, \quad \text{i.e.} \quad T_t \varphi = e^{-\lambda t} \varphi,
\quad \text{for every} \ \ t>0.
\end{align}
The problem of existence and other properties of negative eigenvalues for non-local Schr\"odinger operators
has been widely studied \cite{bib:W74, bib:W75, bib:CMS90, bib:FLS, bib:HL, bib:LS}.

For the reader's convenience we now recall a standard sufficient condition for the existence of negative bound
states (eigenfunctions for negative eigenvalues) for non-local Schr\"odinger operators. The following proposition
is a direct consequence of a basic result due to Weder \cite[Ths 3.6-3.7]{bib:W75} and standard facts on self-adjoint
operators \cite[Prop. 10.4 and 12.8]{bib:Schm}, therefore we omit its proof. Recall
that $L$ and $\cE$ are defined by \eqref{def:gen}-\eqref{def:qform}, and that $L^p + L^{\infty}_{\varepsilon}(\R^d)$,
$p \geq 1$, denotes the space of real-valued functions $V$ on $\R^d$ such that for every $\varepsilon >0$ there exist
$V_{1,\varepsilon} \in L^p(\R^d)$ and $V_{2,\varepsilon} \in L^{\infty}(\R^d)$ with $\left\|V_{2,\varepsilon}\right\|_
{\infty} < \varepsilon$ for which $V = V_{1,\varepsilon} + V_{2,\varepsilon}$.

\begin{proposition} \label{prop:weder}
Let $H_0=-L$ be the pseudo-differential operator with symbol $\psi$ given by \eqref{eq:Lchexp}, and $V = V_+ - V_{-}$ be a Borel
function on $\R^d$ satisfying the following properties:
\begin{itemize}
\item[(1)]
$V_{\pm}\in L^p + L^{\infty}_{\varepsilon}(\R^d)$, for some $p \in [1,\infty)$
\item[(2)]
$V_{-}$ is relatively bounded with respect to $H_0$ with relative bound strictly less than $1$ in quadratic form sense,
i.e., there exist constants $C_{22} \in (0,1)$ and $C_{23} \geq 0$ such that
\begin{align} \label{eq:rfbdd}
\|V_{-}^{1/2} f\|_2^2 \leq C_{22} \cE(f,f) + C_{23} \|f\|_2^2, \quad f \in D(\cE).
\end{align}
\end{itemize}
Then the non-local Schr\"odinger operator $H =H_0+V$ can be defined as a self-adjoint operator in form sense.
Moreover, we have $\Spec H = \Spec_{\rm ess} H \cup \Spec_{\rm d}H$, where $\Spec_{\rm ess} H = \Spec_{\rm ess} H_0 =
[0 ,\infty)$, $\Spec_{\rm d} H \subset (-\infty, 0)$, and $\Spec_{\rm d} H$ consists of isolated eigenvalues of finite
multiplicity whenever it is non-empty. In particular, if there exists $f \in D(\cE)$ such that
\begin{align} \label{eq:suf_ex}
\cE(f,f) + \int_{\R^d} V(x)f^2(x)dx <0,
\end{align}
then $\Spec_{\rm d} H \neq \emptyset$.
\end{proposition}

\begin{remark}
{\rm
\hspace{100cm}
\begin{itemize}
\item[(1)]
It is known that \eqref{eq:Katoclass} is equivalent to the property that
$\lim_{\eta \to \infty} \left\|(-L+\eta \Id)^{-1}|V|\right\|_{\infty} = 0$ (see e.g. \cite[Prop. 3.4]{bib:GrzSz}). Moreover,
a standard argument based on Stein's interpolation theorem (see e.g. \cite[Prop. 3.35]{bib:LHB}) gives that the latter
property implies \eqref{eq:rfbdd} with arbitrarily small $C_{22}$ (i.e., infinitesimal form boundedness of $V$ with
respect to $H_0=-L$). In particular, for any decaying potential both of assumptions (1) and (2) in Proposition
\ref{prop:weder} hold.
\item[(2)]
For decaying potentials the number of negative eigenvalues is bounded from above in many cases of interest, and bounds
are given by variants of the Cwikel-Lieb-Thirring inequality; for the class of non-local Schr\"odinger operators
with Bernstein functions of the Laplacian see \cite{bib:HL}. Whenever \eqref{eq:suf_ex} holds and the number of
negative eigenvalues is finite, a unique ground state exists.
\item[(3)]
Let $0 \leq V_0 \in \cK^X$, $V_0 \neq 0$, and $V_{a,b}(x):= - a V_0(bx)$, with $a, b    >0$. Notice that for any non-zero
$f \in D(\cE)$ there are appropriate choices of $a, b$ for which $\cE(f,f) + \int_{\R^d} V_{a,b}(x)f^2(x)dx <0$.
In particular, if $V_0$ is a non-increasing radial function and \eqref{eq:suf_ex} holds for $V_{a_0,b_0}$ with
some $a_0, b_0$, then this extends to all $a > a_0$ and $b \in (0,b_0)$. Moreover, let $f_r$ and $\mu_r$ be a
ground state and ground state eigenvalue, respectively, for an isotropic unimodal L\'evy process with diffusion matrix
$\left\|A\right\| \Id$ and radially non-increasing L\'evy density $g(|x|)$, killed on leaving the ball $B(0,r)$, $r>0$
(recall that $g$ is determined by (A1)). Denote the corresponding quadratic form  by $(\cE_0, D(\cE_0))$. Clearly,
$\supp f_r = B(0,r)$, $\left\|f_r\right\|_2=1$, and it is seen that $f_r$ is a non-increasing radial function
\cite[Def.1.3, Cor.2.3]{bib:AL}, see also \cite[Lem.3.1]{bib:Kal}. Since $\cE(f_r,f_r) \leq C_5 \cE_0(f_r,f_r) =
C_5 \mu_r$, we have
$$
\cE(f_r,f_r) + \int_{\R^d} V_{a,b}(x)f_r^2(x)dx
\leq
C_5 \mu_r - a \int_{B(0,r)} V_0(bx) f_r^2(x) dx
\leq
C_5 \mu_r - a V_0(br),
$$
for $r>0$. In many cases good approximations of $\mu_r$ are known (see, e.g., \cite{bib:CS0, bib:KKM}).
With this, for given $C_5$, $a$ and $b$, we can settle if the right hand side is
negative and even estimate its distance from zero. On the other hand, when the underlying L\'evy process is
recurrent and $V \not \equiv 0$ is a non-positive bounded potential with compact support, then a negative bound state
exists \cite[Th. V.1]{bib:CMS90}.
\end{itemize}
}
\end{remark}

\begin{example}
{\rm The following decaying potentials are some possible choices and applications.
\begin{itemize}
\item[(1)]
\emph{Potential wells:} Let $V(x)= - v(x)$ with a compactly supported, non-negative bounded Borel function $v \not \equiv 0$.
Specifically, we can choose $V(x)=-a \1_{B(0,1)}(bx)$, for $a, b>0$.

\vspace{0.1cm}
\item[(2)]
\emph{Coulomb-type potentials:} Let $g$ in Assumption (A1) be such that $g(r)=r^{-d-\alpha}$, $r \in (0,1]$, for some $\alpha
\in (0,2)$, and let $V(x)=-(a_1|x|^{-\beta_1} \wedge a_2|x|^{-\beta_2})$, with $\beta_1 \in (0,\alpha \wedge d]$, $\beta_2 \in
[\beta_1, \infty)$ and $a_1, a_2 >0$. Then $V \in \cK^X$ can be directly checked by using \eqref{eq:Kato_new} and
\cite[Th. 2]{bib:KS14}.

\vspace{0.1cm}
\item[(3)]
\emph{Yukawa-type potentials:}
Let $g$ in Assumption (A1) be as in (2) above and let $V(x)=-(a_1|x|^{-\beta_1} \wedge a_2|x|^{-\beta_2}e^{-b|x|})$,
with $\beta_1 \in (0,\alpha \wedge d]$, $\beta_2 \in [\beta_1, \infty)$
and $a_1, a_2, b >0$.

\vspace{0.1cm}
\item[(4)]
\emph{P\"oschl-Teller potential:} This is the case of $V(x)= -a/\cosh^2 (b|x|)$ with $a, b > 0$.

\vspace{0.1cm}
\item[(5)]
\emph{Morse potential:} This is the case of $V(x)= a((1-e^{-b(|x|-r_0)})^2-1)$ with $a, b, r_0 > 0$.
\end{itemize}
}
\end{example}

\medskip
\noindent
We will everywhere below assume that every eigenfunction $\varphi$ is normalized so that $\|\varphi\|_{2} = 1$.
Moreover, by the assumption $V \in \cK_{\pm}^X$ we have $T_t(L^2(\R^d)) \subset L^{\infty}(\R^d)$ and $T_t(L^{\infty}(\R^d))
\subset C_{\rm b}(\R^d)$ for every $t>t_{\rm b}$, compare Lemma \ref{lm:semprop}. Therefore $\varphi = e^{\lambda t} T_t\varphi
\in C_{\rm b}(\R^d)$, in particular, it makes sense to study pointwise estimates of $\varphi$.

When $\lambda_0:= \inf \Spec H$ is an isolated eigenvalue (i.e., a ground state exists), by standard arguments based on Lemma
\ref{lm:semprop} (4) and \cite[Th. XIII.43]{bib:RS} it follows that it is unique and the corresponding eigenfunction $\varphi_0$
has a strictly positive version, which will be our choice throughout. It is known that whenever a ground state at eigenvalue
$\lambda_0 \neq 0$ exists, $\varphi_0$ is the only non-negative eigenfunction of $H$ corresponding to a non-zero eigenvalue.

Below we will often use the following resolvent representation of eigenfunctions. Let $\lambda$
and $\varphi$ be such that
$T_t \varphi(x) = e^{-\lambda t} \varphi(x)$ for every $x \in \R^d$ and $t>0$. Choose $\theta
\in \R$ such that $\theta + \lambda >0$. By integrating on both sides of the equality
$$
e^{-(\theta+\lambda)t} \varphi(x) = \ex^x\left[e^{-\int_0^t(\theta+V(X_s))ds} \varphi(X_t)\right],
\quad t>0, \; x \in \R^d,
$$
we obtain
\begin{align}\label{eq:eig}
\varphi(x) = (\theta+\lambda)
\int_0^{\infty}\ex^x\left[ e^{-\int_0^t(\theta+V(X_s))ds} \varphi(X_t) \right]dt, \quad x \in \R^d.
\end{align}
Notice that by combining this with \eqref{eq:pot1} applied to $f=\varphi$ for an arbitrary open set
$D \subset \R^d$ and $x \in D$, we readily obtain
\begin{align} \label{eq:eig1}
\varphi(x)  = (\theta+\lambda) \ex^x\left[ \int_0^{\tau_D} e^{-\int_0^t(\theta+V(X_s))ds} \varphi(X_t)
dt\right] +
\ex^x\left[\tau_D < \infty; e^{-\int_0^{\tau_D}(\theta+V(X_s))ds} \varphi(X_{\tau_D}) \right]
\end{align}
which will be a key formula in what follows.

We close this subsection by showing that positive eigenfunctions (in particular, ground states) are bounded from
below by $\nu$, no matter what the absolute value of $\lambda$ is. This fact provides a general reference point
and justifies the question how far is the decay of eigenfunctions from the L\'evy intensity of the given process. We
also show that if (A1.2) or (A1.3) fails to hold, then such eigenfunctions cannot be bounded from above by $\nu$,
even when the absolute value of $\lambda$ is large. In the next subsections we will study upper bounds of the
eigenfunctions.

\begin{theorem}[\textbf{Lower bound and necessary condition for upper bound}]
\label{prop:lowerbound}
Let $\pro X$ be a symmetric L\'evy process with L\'evy-Khinchin exponent satisfying \eqref{eq:Lchexp}-\eqref{eq:nuinf},
and let assumption (A4) hold. Suppose that $\varphi \in L^2(\R^d)$ is a positive eigenfunction at eigenvalue $\lambda
\in \R$. Then the following hold.
\begin{itemize}
\item[(1)]
If (A1.1)-(A1.2) are satisfied, then for every $\delta  >0$ there exists $r=r(V,\delta) \geq 1$ such that
$$
\varphi(x) \geq K\, \nu(x), \quad |x|   > r + 1,
$$
where $K:= \frac{1-e^{-(|\lambda|+\delta)}}{C_5^2 C_6^{\left\lceil r\right\rceil + 1} (|\lambda|+\delta)} \, \pr^0(\tau_{B(0,1)} > 1) \int_{B(0,r)}
\varphi(z)dz$.

\medskip
\item[(2)]
Let (A1.1) be satisfied and consider the following two disjoint cases.
\begin{itemize}
\item[(i)]
(A1.2) holds and (A1.3) does not hold.
\item[(ii)]
(A1.2) does not hold (and hence also (A1.3) does not hold).
\end{itemize}
Then in either of cases (i) and (ii) we have
$$
\limsup_{|x| \to \infty} \frac{\varphi(x)}{\nu(x)}= \infty.
$$
\end{itemize}
\end{theorem}

\begin{proof}
Take any $\delta>0$. Note that by assumption (A4) there exists $r \geq 1$ such that $\sup_{|y| \geq r} |V(y)| \leq
\delta/2$. Also, let $\theta = |\lambda| + \delta/2$. With this, an application of \eqref{eq:eig1} to $D=\overline B(0,r)^c$
gives for every $|x| > r$
$$
\varphi(x) \geq \ex^x\left[\tau_{\overline B(0,r)^c} <
\infty; e^{- (\theta+\delta/2) \tau_{\overline B(0,r)^c}} \varphi(X_{\tau_{\overline B(0,r)^c}}) \right].
$$
Both assertions (1) and (2) follow now by direct application of Theorem \ref{prop:lowerboundharm} to the function $f(y)=\varphi(y)$,
$|y| \leq r$, $f(y)=\ex^{y}\left[\tau_{\overline B(0,r)^c} < \infty; e^{- (|\lambda| + \delta) \tau_{\overline B(0,r)^c}}
\varphi(X_{\tau_{\overline B(0,r)^c}}) \right]$, $|y|  > r$, noticing that it is regular $(X,|\lambda| + \delta)$-harmonic in
$\overline B(0,r)^c$.
\end{proof}

\subsection{Upper bound: cases of sufficiently low-lying and arbitrary negative eigenvalues}

In this subsection we state our main results on the upper bounds of eigenfunctions in their most general form. They
are consequences of the estimates for harmonic functions obtained in Section 3 and will be illustrated and discussed
in detail in the next subsection, where we analyze specific classes and examples of processes of interest.

The first result says that whenever an eigenvalue is sufficiently low-lying below zero with respect to the
given jump-paring L\'evy process, the corresponding eigenfunction is dominated by $\nu$ at infinity.

\begin{theorem} [\textbf{Low-lying negative eigenvalues}]
\label{thm:defic6}
Let $\pro X$ be a L\'evy process with L\'evy-Khinchin exponent $\psi$ given by \eqref{eq:Lchexp}-\eqref{eq:nuinf}
and let Assumptions (A1)--(A4) hold. Denote
\begin{align} \label{eq:eta0}
\eta_0(X):= 2 \, C_5^4 \, h_1(X, 1, 2) \, K_1^X(2)  + h_2(X, 1) \, |B(0,2)| \, K_3^X(2, 3, \infty).
\end{align}
Suppose that $\varphi \in L^2(\R^d)$ is an eigenfunction for an eigenvalue $\lambda \in (-\infty, -\eta_0)$.
Then there exist $C_{25} = C_{25}(X,\lambda)$ and $R=R(X,\lambda) >0$ such that
$$
|\varphi(x)| \leq C_{25} \left\|\varphi\right\|_{\infty} \nu(x), \quad |x| \geq R.
$$
\end{theorem}

\begin{proof}
Choose $\theta > |\lambda|$. Let $\varepsilon >0$ be small enough such that $|\lambda|-\varepsilon > \eta_0$, and $r > 0$ be large enough such that
$\sup_{|y| \geq r}|V(y)| \leq \varepsilon$. Denote $\tau_r := \tau_{\overline B(0,r)^c}$. By \eqref{eq:eig1} applied to
$D = \overline B(0,r)^c$, for every $|x| > r$ we have
\begin{align*}
|\varphi(x)|
& \leq
(\theta-|\lambda|) \ex^x\left[ \int_0^{\tau_r} e^{-(\theta-\varepsilon)t} |\varphi(X_t)| dt\right] +
\ex^x\left[\tau_r < \infty; e^{-\int_0^{\tau_r}(\theta+V(X_s))ds} |\varphi(X_{\tau_r})| \right] \\
& \leq
\left\|\varphi\right\|_{\infty} \left(\frac{\theta-|\lambda|}{\theta-\varepsilon} +
\ex^x\left[e^{-(\theta-\varepsilon) \tau_r} \right] \right).
\end{align*}
Thus by taking the limit $\theta \to |\lambda|$ we obtain
$$
|\varphi(x)| \leq \left\|\varphi\right\|_{\infty}\ex^x\left[e^{-(|\lambda|-\varepsilon) \tau_r}
\right], \quad |x|    > r.
$$
To complete the proof it remains to apply Lemma \ref{lem:defic1} to the function $f$ as in \eqref{eq:modelhf}
which is $(X,\eta)$-harmonic in $\overline B(0,r)^c$ with $\eta=|\lambda|-\varepsilon$. Indeed,
by \eqref{eq:eta0} its assumptions are now satisfied with $r_1=1$, $r_2 = 2$ and $r_3 = 3$.
\end{proof}
\noindent
Note that in fact in the above theorem it suffices to assume that (A3) holds for $R \leq 1$ only.

The next result gives sufficient conditions according to a slow or fast decay of the L\'evy
intensity under which the eigenfunctions corresponding to arbitrary eigenvalues $\lambda < 0$
are bounded above at infinity by $\nu$. Recall the discussion of the specific conditions below as
done before Theorems \ref{thm:defic2}-\ref{thm:defic3}.

\begin{theorem} [\textbf{Arbitrary negative eigenvalues}]
\label{thm:defic4}
Let $\pro X$ be a L\'evy process with L\'evy-Khinchin exponent $\psi$ given by \eqref{eq:Lchexp}-\eqref{eq:nuinf}
and let Assumptions (A1)--(A4) hold. Suppose that $\varphi \in L^2(\R^d)$ is an eigenfunction at the eigenvalue
$\lambda < 0$. Consider the following cases.
\begin{itemize}
\item[(1)]
\texttt{L\'evy intensities with slow decay at infinity}: conditions \eqref{eq:nu_dbl} and
\eqref{eq:Green_bdd} hold.
\item[(2)]
\texttt{L\'evy intensities with fast decay at infinity}: conditions \eqref{eq:fsm}-\eqref{eq:comp_dens},
\eqref{eq:suf1} and \eqref{eq:intr_killing}-\eqref{eq:unif_bdd} hold.
\end{itemize}
In both cases (1) and (2) above there exist $C_{24}= C_{24}(X,\lambda)$ and $R=R(X,\lambda) >0$ such that
$$
|\varphi(x)| \leq C_{24} \left\|\varphi\right\|_{\infty} \nu(x), \quad |x| \geq R.
$$
\end{theorem}

\begin{proof}
Let $\lambda$ be as in the assumption. Take $\varepsilon
>0$ small enough so that $|\lambda|-\varepsilon > 0$, and $r > 0$ large enough such that
$\sup_{|y| \geq r}|V(y)| \leq \varepsilon$. Denote $\tau_r := \tau_{\overline B(0,r)^c}$.
By the same argument as in the proof of Theorem \ref{thm:defic6} above we get
\begin{align*}
|\varphi(x)| \leq \left\|\varphi\right\|_{\infty} \ex^x\left[e^{-(|\lambda|-\varepsilon) \tau_r} \right], \quad |x| > r.
\end{align*}
The claimed bound now follows for cases (1) and (2) by an application of Theorems \ref{thm:defic2} and \ref{thm:defic3}
with $\eta=|\lambda|-\varepsilon$, respectively.
\end{proof}
\noindent
Recall that whenever $\pro X$ is an isotropic unimodal jump-paring L\'evy process in $\R^d$, $d \geq 3$, conditions
\eqref{eq:Green_bdd}, \eqref{eq:suf1}, \eqref{eq:comp_dens} automatically hold, and the assumptions above dividing
into specific subclasses reduce to \eqref{eq:nu_dbl} in case (1), and to \eqref{eq:fsm} and the basic conditions
\eqref{eq:intr_killing}-\eqref{eq:unif_bdd} in case (2). As it will be seen below, assumptions
\eqref{eq:intr_killing}-\eqref{eq:unif_bdd} are essential and sharp in the sense that for a large class of
processes they actually provide a necessary and sufficient condition for the ground state to be comparable to the
L\'evy density $\nu$ at infinity (see discussion preceding Proposition \ref{prop:sharp}).

\subsection{Specific cases}
\label{subsec:conseq}
Throughout this subsection we assume that $\lambda < 0$. As it follows from the results above, the behaviour of the
eigenfunctions at infinity significantly depends on the relative position of the corresponding eigenvalue from the
edge of the continuous spectrum and on the decay properties of the L\'evy intensity. To illustrate this in some explicit
detail, consider the specific choice
\begin{align} \label{eq:psi_ex}
\psi(\xi) = a |\xi|^2+ \int (1-\cos(\xi \cdot z)) \nu(z) dz, \quad \xi \in \R^d,
\end{align}
of the characteristic exponent (\ref{eq:Lchexp}), where $a \geq 0$, $\nu(dx)=\nu(x)dx$ is such that
$\nu(\R^d \backslash \left\{0\right\}) = \infty$, $\nu(x)=\nu(-x)$, and there is a non-increasing function
$g:(0,\infty) \to (0,\infty)$ such that
\begin{align}  \label{eq:nu_g}
C_{26} g(|x|) \leq \nu(x) \leq C_{27} g(|x|), \quad x \in \R^d \backslash \left\{0\right\},
\end{align}
with constants $C_{26} \in (0,1]$ and $C_{27} \in [1, \infty)$. Moreover, whenever $C_{26}=C_{27}=1$
(i.e., $\nu$ is a non-increasing radial function), we impose the regularity condition
\begin{align}
\label{eq:g_low}
g(r) \geq \frac{C_{28}}{r^d}, \quad r \in (0,1],
\end{align}
on the small jumps with a constant $C_{28}$. On the other hand, when $C_{26} < 1$ or $C_{27}   >1$, we assume
that \eqref{eq:weakscaling} holds. In the proofs below we will often use the fact that \eqref{eq:g_low} and \eqref{eq:weakscaling}
imply $\liminf_{|\xi| \to \infty} \psi(\xi)/\log |\xi| > 0$, which guarantees that assumption (A2) holds. 
Indeed, observe that by \cite[Lem. 5 (a)]{bib:KS14}, \eqref{eq:PruitH} and \eqref{eq:nu_g},
$$
\psi(\xi) \asymp \Psi(|\xi|) \geq C_1 H(1/|\xi|) \geq C_1 a|\xi|^2 + C_1 C_{26}\int_{1/|\xi| < |y|< 1} g(|y|)dy, \quad |\xi| > 1.
$$
From this we see that $\psi(\xi) \geq C(a |\xi|^2 + \log|\xi|)$ and $\psi(\xi) \geq C(a |\xi|^2 + |\xi|^{\gamma_1})$, $|\xi| > 1$, with some $C>0$ under \eqref{eq:g_low} and \eqref{eq:weakscaling}, respectively. In particular, $p(t,0) = \int_{\R^d} e^{-t \psi(\xi)} d \xi < \infty$, for sufficiently large $t >0$ and (A2) holds. Clearly, when $a>0$, the same is true even without \eqref{eq:g_low} and \eqref{eq:weakscaling}.

\begin{corollary} [\textbf{Polynomially decaying L\'evy intensities}]
\label{thm:polynomial}
Let $\pro X$ be a L\'evy process with L\'evy-Khinchin exponent $\psi$ given by \eqref{eq:psi_ex}--\eqref{eq:nu_g}.
Also, if \eqref{eq:nu_g} holds with $C_{26}=C_{27}=1$, then assume \eqref{eq:g_low}, if it holds with other values
of constants, then assume \eqref{eq:weakscaling}. Let the profile $g$ satisfy
\begin{align}
\label{eq:polynomial}
g(r) = r^{-d-\delta}, \quad r \geq 1, \ \ \text{with} \ \ \delta >0.
\end{align}
If Assumption (A4) holds and $\varphi \in L^2(\R^d)$ is an eigenfunction corresponding to the eigenvalue $\lambda < 0$,
then there exist $C_{29}=C_{29}(X,\lambda)$ and $R=R(X,\lambda) \geq 1$ such that
$$
|\varphi(x)| \leq C_{29} \left\|\varphi\right\|_{\infty} |x|^{-d-\delta}, \quad |x| \geq R.
$$
Moreover, if $\varphi = \varphi_0$ is a ground state, then $\varphi_0(x) \asymp |x|^{-d-\delta}$, $|x| \geq R$.
\end{corollary}

\begin{proof}
First observe that (A1.1) is assumed by default, and the remaining conditions (A1.2)-(A1.3) easily follow from
\eqref{eq:polynomial}. Assumption (A2) is a consequence of \eqref{eq:g_low} or \eqref{eq:weakscaling}, and
\eqref{eq:Green_bdd} follows from \eqref{eq:suf2} and \eqref{eq:suf1} as in Lemma \ref{lem:ver_cond4}. Indeed, by
assuming \eqref{eq:polynomial} and by directly checking that
$$
\Psi(r) \asymp \left\{
\begin{array}{lrl}
r^{\delta}   & \mbox{  for  } & 0 < \delta < 2,
\vspace{0.2cm} \\
r^2 \log(1/r)   & \mbox{  for  } & \delta = 2,
\vspace{0.2cm} \\
r^2   & \mbox{  for  } & \delta  > 2,
\end{array} \right. \quad r \in (0,1],
$$
condition \eqref{eq:suf2} with $\Phi=\Psi$ can be established similarly as in \eqref{eq:aux_est_1}.
Moreover, whenever $C_{26}=C_{27}=1$, \eqref{eq:suf1} holds automatically for all $r >0$.
Otherwise, due to \eqref{eq:weakscaling} and \eqref{eq:polynomial},
the same holds by Proposition \ref{prop:bound_by_psi}. Specifically,
when $\delta >2$, \eqref{eq:suf1} follows directly from the statement of this result, and when $\delta \in (0,2]$,
this can be obtained by following through the argument in the proof. Assumption (A3) holds
as well. Secondly, notice that \eqref{eq:nu_dbl} immediately follows from \eqref{eq:polynomial}. The proof of the upper
bound can be completed by an application of Theorem \ref{thm:defic4} in case (1). The lower bound on $\varphi_0$ follows
directly by Theorem \ref{prop:lowerbound}.
\end{proof}

\begin{remark} \label{rem:poly}
{\rm
The assumptions of the above theorem cover the cases when
$$
g(r) = r^{-d-\gamma} \1_{\left\{r \in (0,1]\right\}} + r^{-d-\delta} \1_{\left\{r \geq 1\right\}}, \quad \mbox{with}
\quad \gamma \in [0, 2) \ \ \text{and} \ \ \delta > 0,
$$
and $\nu(x)=g(|x|)$ for $\gamma \in [0, 2)$, or more generally, $\nu(x) \asymp g(|x|)$ for $\gamma \in (0, 2)$. For $a=0$,
the first class includes some important cases of subordinate Brownian motion and isotropic unimodal L\'evy processes like
the isotropic $\alpha$-stable process ($\gamma=\delta=\alpha \in (0,2)$), their mixtures of stability indices $\alpha_1,
...,\alpha_n$ ($\gamma= \min_i \alpha_i$, $\delta = \max_i \alpha_i$), geometric $\alpha$-stable process ($\gamma=0$,
$\delta = \alpha \in (0,2)$), layered $\alpha$-stable processes ($\gamma = \alpha \in (0,2)$, $\delta > 2$), and many
others. All of the above examples can also be considered with a non-zero Brownian component ($a>0$), which covers a wide
class of jump-diffusions. The second case, where $\nu$ is not a strictly radial function, allows a variety of
polynomial-type processes like symmetric stable-like, strictly stable, and other L\'evy processes with jump intensities
comparable to the cases of isotropic processes as above.
}
\end{remark}

Next we consider the case when the L\'evy intensity is lighter than polynomial but heavier than exponential at infinity.

\begin{corollary} [\textbf{Sub-exponentially decaying L\'evy intensities}]
\label{thm:subexp}
Let $\pro X$ be a L\'evy process with L\'evy-Khinchin exponent $\psi$ given by \eqref{eq:psi_ex}--\eqref{eq:nu_g}.
Moreover, if \eqref{eq:nu_g} holds with $C_{26}=C_{27}=1$, then assume \eqref{eq:g_low}, if it holds with other values
of constants, then assume \eqref{eq:weakscaling}. Also, let the profile $g$ satisfy
\begin{align} \label{eq:subexp}
g(r) = e^{-c r^{\beta}} r^{-\delta}, \quad r \geq 1, \quad \text{with} \ \ c>0, \ \beta \in (0,1) \ \ \text{and}
\ \ \delta \geq 0.
\end{align}
If Assumption (A4) holds and $\varphi \in L^2(\R^d)$ is an eigenfunction at the eigenvalue $\lambda < 0$, then there exist
$C_{30}=C_{30}(X,\lambda)$ and $R=R(X,\lambda) \geq 1$ such that
$$
|\varphi(x)| \leq C_{30} \left\|\varphi\right\|_{\infty} e^{-c|x|^{\beta}}|x|^{-\delta}, \quad |x| \geq R.
$$
Moreover, if $\varphi = \varphi_0$ is a ground state, then $\varphi_0(x) \asymp e^{-c|x|^{\beta}}|x|^{-\delta}$, $|x| \geq R$.
\end{corollary}

\begin{proof}
The upper bound follows from Theorem \ref{thm:defic4} in case (2). Assumptions (A1.1)-(A1.2) hold by default, (A1.3) by
\cite[Prop. 2]{bib:KS14}, and (A2) follows from $\liminf_{|\xi| \to \infty} \psi(\xi)/\log |\xi|    > 0$, which is the
case under assumptions \eqref{eq:g_low} or \eqref{eq:weakscaling}. Also, \eqref{eq:fsm} is immediate and implies
\eqref{eq:suf2}. With this, whenever $C_{26}=C_{27}=1$,  assumptions \eqref{eq:comp_dens}, \eqref{eq:suf1} and (A3)
are satisfied automatically, otherwise they follow from Proposition \ref{prop:bound_by_psi}. Thus it suffices to verify the
remaining conditions \eqref{eq:intr_killing}-\eqref{eq:unif_bdd}.

First we show that there exist constants $c_1, c_2$ such that for every $s_2 > 2s_1 \geq 2$
$$
K^X_1(s_2) \leq c_1 s_2^{d-\delta-\beta} e^{-c(1-\beta){s_2}^{\beta}} \quad \text{and} \quad K^X_2(s_1, s_2, \infty) \leq
c_2 e^{\frac{c \beta s_1}{(s_2-s_1)^{1-\beta}}}.
$$
For the first bound, observe that for every $x,y \in \R^d$, $y \neq 0$, $x \neq y$, we have by Lagrange's theorem
\begin{align*}
|y|^{\beta} + |x-y|^{\beta} & = (|y| \vee |x-y|)^{\beta} + (|y| \wedge |x-y|)^{\beta} \\
& \geq (|y| + |x-y|)^{\beta} - \frac{\beta(|y| \wedge |x-y|)}{(|y| \vee |x-y|)^{1-\beta}}
+ (1-\beta)(|y| \wedge |x-y|)^{\beta} + \beta(|y| \wedge |x-y|)^{\beta} \\
& \geq |x|^{\beta} - \frac{\beta(|y| \wedge |x-y|)}{(|y| \vee |x-y|)^{1-\beta}}
+ (1-\beta)(|y| \wedge |x-y|)^{\beta} + \frac{\beta(|y| \wedge |x-y|)}{(|y| \wedge |x-y|)^{1-\beta}} \\
& \geq |x|^{\beta} + (1-\beta)(|y| \wedge |x-y|)^{\beta}.
\end{align*}
Thus, by \eqref{eq:subexp} and \eqref{def:K1}, we obtain
$$
K^X_1(s_2) \leq c_3 \int_{|y|>s_2} e^{-c(1-\beta)|y|^{\beta}} |y|^{-\delta} dy,
$$
which gives the required bound. A similar argument gives $|x|^{\beta}-|x-y|^{\beta} \leq \beta s_1/(s_2-s_1)^{1-\beta}$ for all $|x| \geq s_2$ and $|y| \leq s_1$, and thus
$$
g(|x-y|) = g(|x|) \frac{g(|x-y|)}{g(|x|)} \leq e^{\frac{c \beta s_1}{(s_2-s_1)^{1-\beta}}}
\left(1+\frac{s_1}{s_2-s_1} \right)^{\delta} g(|x|)   , \quad |x| \geq s_2, \ \ \ |y| \leq s_1,
$$
which yields the claimed bound for $K^X_2$. In particular, for $\kappa_1 \geq 2$ we have
$$
K^X_1(\kappa_1 s_1) K^X_2(s_1, \kappa_1 s_1, \infty) \leq c_1 c_2 (\kappa_1 s_1)^{d-\delta-\beta}e^{-c \left((1-\beta) \kappa_1^{\beta} - \frac{\beta }
{(\kappa_1-1)^{1-\beta}} \right) {s_1}^{\beta}}, \quad s_1 \geq 1, \ \ \kappa_1 \geq 2,
$$
which implies \eqref{eq:intr_killing} for every $\kappa_1 > 2 \vee (\beta/(1-\beta))^{1/\beta}$. Condition
\eqref{eq:unif_bdd} with $\kappa_2 \geq c_2$ easily follows from the upper bound of $K^X_2$ established above. The lower
bound on $\varphi_0$ is a direct consequence of Theorem \ref{prop:lowerbound}.
\end{proof}

\begin{remark}
\label{rem:inter}
{\rm
With some lengthy but straightforward computations we can show that the decay of $\nu(x)$ at infinity interpolating between
the polynomial and sub-exponential cases such as $|x|^{-\log|x|}$ leads to similar bounds of eigenfunctions in terms of
$\nu(x)$ as in Corollaries \ref{thm:polynomial}-\ref{thm:subexp}. The same applies for $\nu(x) \asymp e^{-|x|/\log|x|}$, $|x|
\to \infty$. Note that this decay rate is even closer to the strictly exponential case than in Corollary \ref{thm:subexp}.
The details are left to the reader.
}
\end{remark}

\subsection{Phase transition in the decay rates and strongly tempered L\'evy intensities}
\label{subsec:phase}

As seen in the previous subsection, the fall-off of the ground state is determined by the L\'evy intensity $\nu$ as long as this
decays strictly sub-exponentially. We now describe a qualitative change (``phase transition") in the decay rate behaviour. We show
that if $|\lambda|$ is small with respect to the spectral edge of the free process and $\nu$ decreases exponentially, then the decay
rate of $\varphi_0$ is slower than the decay rate of $\nu$, with essential contribution of $\lambda$. However, when $|\lambda|$ is large enough, the
fall-off is again dominated by $\nu$ as long as the basic jump-paring condition (A1.3) holds (recall that this implies (A1.2)). In
particular, this partly includes the exponential case. On the other hand, when (A1.2) or (A1.3) fails, then as proved in Theorem
\ref{prop:lowerbound}, the fall-off of $\varphi_0$ is slower than the decay of $\nu$, no matter how large $|\lambda_0|$ is. This class includes not only super-exponentially decaying
L\'evy intensities, but also cases of exponentially decaying L\'evy intensities for which the jump paring condition does not hold, so
the exponential case falls on the dividing line in a peculiar way.

First we consider the exponential case.
\begin{corollary}  [\textbf{Exponentially decaying L\'evy intensities}]
\label{prop:exp}
\noindent
Let $\pro X$ be a L\'evy process with L\'evy-Khinchin exponent $\psi$ given by \eqref{eq:psi_ex}--\eqref{eq:nu_g} with $a=0$
and let Assumption (A4) hold.
Moreover, let \eqref{eq:weakscaling} be satisfied and assume
\begin{align} \label{eq:exp}
g(r) = e^{-c r} r^{-\delta}, \quad r \geq 1, \quad \text{with} \ \ c>0 \ \ \text{and} \ \ \delta \geq 0.
\end{align}
Then we have the following:
\begin{itemize}
\item[(1)]
If a ground state at eigenvalue $\lambda_0 < 0$ exists, then there exist
$C_{31} >0$ and $R>0$ such that
$$
\varphi_0(x) \geq C_{31} \, e^{-c |x|}|x|^{-\delta},
\quad |x| \geq R.
$$
\item[(2)]
If $\delta > (d+1)/2$ (i.e. (A1.3) holds), then for every eigenfunction $\varphi$ corresponding to the eigenvalue
$\lambda \in (-\infty, -\eta_0)$ there exists a constant $C_{32}>0$ and $R>0$ such that
$$
|\varphi(x)| \leq C_{32} \left\|\varphi\right\|_{\infty} e^{-c|x|} |x|^{-\delta}, \quad |x| \geq R.
$$
In particular, if $\delta > (d+1)/2$ and  $\lambda_0 \in (-\infty, -\eta_0)$, then $\varphi_0(x) \asymp e^{-c|x|}
|x|^{-\delta}$, $|x| \geq R$.
\item[(3)]
If \eqref{eq:weakscaling} holds with $\gamma_1=\gamma_2$ and a ground state at eigenvalue $\lambda_0 < 0$ exists, then there is a constant
$C_{33}    > c$ such that for every $\varepsilon >0$ there exist $C_{34} >0$ and $R>0$ for which
$$
\varphi_0(x) \geq C_{34} \left(e^{- C_{33} \sqrt{|\lambda_0|+\varepsilon} |x|} \ \vee \ e^{-c |x|}|x|^{-\delta} \right),
\quad |x| \geq R.
$$
\end{itemize}
\end{corollary}

\begin{proof}
The assumptions (A1.1)-(A1.2), (A2) and (A3) follow by similar arguments as above. Assumption (A1.3) holds if and
only if $\delta > (d+1)/2$ \cite[Prop. 2]{bib:KS14}. The lower bound in terms of $e^{-c|x|} |x|^{-\delta}$ in
(1)-(3) is a direct consequence of Theorem \ref{prop:lowerbound}. Similarly, the upper bound in (2) follows from
Theorem \ref{thm:defic6}. It suffices to show that if \eqref{eq:weakscaling} holds with $\gamma_1=\gamma_2$, then $\varphi_0(x) \geq C_{34} e^{- C_{33} \sqrt{|\lambda_0|
+\varepsilon} |x|}$, for sufficiently large $|x|$. Notice that we only need to consider the case $|\lambda_0| +
\varepsilon \in (0,1]$. By (1.14) in \cite[Th. 1.2]{bib:CKK}, for every $\eta \in (0,1]$ and $|x| \geq 1$ we have
$$
G^{\eta}(x) \geq \int_1^{\infty} e^{-\eta t} p(t,x) dt \geq c_1  \int_1^{\infty} e^{-\eta t} e^{-c_2 \left(|x|
\wedge \frac{|x|^2}{t} \right) }t^{-d/2} dt \geq c_1  e^{-c_2 \sqrt{\eta} |x|} \int_{|x|/\sqrt{\eta}}^{\infty}
 e^{-\eta t} t^{-d/2} dt,
$$
with some $c_1 >0$ and $c_2 \geq c$. Moreover,  $\int_{|x|/\sqrt{\eta}}^{\infty} e^{-\eta t} t^{-d/2} dt
\asymp e^{-\sqrt{\eta}|x|} |x|^{-d/2}$.
By the same argument as in the proof of Theorem \ref{thm:defic6} (with the converse inequalities) and using
\eqref{eq:capacitory}, for every $\varepsilon >0$ satisfying $|\lambda_0| + \varepsilon \in (0,1]$ there exist
$r>0$ and $c_3=c_3(X,\eta,r)$ such that
\begin{align*}
\varphi_0(x) & \geq \inf_{|z| \leq r} \varphi_0(z) \ex^x[e^{-(|\lambda_0| + \varepsilon) \tau_{\overline B(0,r)^c}}] \\
& \geq \inf_{|z| \leq r} \varphi_0(z) \mu^{|\lambda_0| + \varepsilon}_{B(0,r)}(\overline B(0,r)) \inf_{|y| \leq r}
G^{|\lambda_0| + \varepsilon}(x-y) \geq c_3 e^{-(1+c_2) \sqrt{|\lambda_0| + \varepsilon} |x|} |x|^{-d/2}, \quad |x| > r+1.
\end{align*}
\end{proof}
\noindent
Note that the condition $\gamma_1=\gamma_2$ in part (3) of the corollary is a technical assumption required only for an application of the large-time estimates of transition densities in \cite{bib:CKK}.

The above result covers a large family of L\'evy processes with jump intensities exponentially localized at infinity.
Important examples to this class are relativistic $\alpha$-stable like processes ($c=m^{1/\alpha}$, $\delta=(d+1+\alpha)/2$,
$\alpha \in (0,2)$) for which (A1.3) holds, and other processes with L\'evy measures with tails for which (A1.3) does not
hold, such as variance gamma like (that is, geometric $2$-stable like) processes ($c=1$, $\delta=(d+1)/2$), and some Lamperti-type
transformed stable processes ($c>0$, $\delta=d-1$).

The next result, taken together with Corollaries \ref{thm:polynomial}, \ref{thm:subexp} and \ref{prop:exp}, shows that
for a large subclass of jump-paring L\'evy processes assumptions \eqref{eq:intr_killing}-\eqref{eq:unif_bdd} give in fact a
necessary and sufficient condition for the result in Theorem \ref{thm:defic4}. In particular, for L\'evy intensities $\nu$
such that $\nu(x) \asymp e^{-c|x|^{\beta}} |x|^{-\delta}$ for $|x| \to
\infty$, with $c>0$, $\beta \in (0,1]$, $\delta \geq 0$,
conditions \eqref{eq:intr_killing}-\eqref{eq:unif_bdd} hold
if and only if $\beta \in (0,1)$, which is equivalent to the property that $\varphi_0 \asymp \nu$. Recall that when the
process is outside the jump-paring class (i.e., $\beta=1$ and $\delta \in [0,(d+1)/2]$ or $\beta>1$ in the case above), $\varphi_0$ decays slower than $\nu$.

\begin{proposition} \label{prop:sharp}
Let $\pro X$ be a L\'evy process with L\'evy-Khinchin exponent $\psi$ given by \eqref{eq:psi_ex}--\eqref{eq:nu_g}
such that \eqref{eq:weakscaling} and \eqref{eq:exp} hold. Whenever $\delta > (d+1)/2$, conditions
\eqref{eq:fsm}-\eqref{eq:comp_dens} and \eqref{eq:suf1} hold. However, there exist constants $C_{35}, C_{36}$ such
that for every $s_2 > 2s_1 \geq 2$ we have
$$
K^X_1(s_2) \geq C_{35} {s_2}^{(d+1)/2 - \delta} \quad \text{and} \quad K^X_2(s_1, s_2, \infty) \geq C_{36} e^{c s_1}.
$$
In particular, \eqref{eq:intr_killing}-\eqref{eq:unif_bdd} fail to hold for any $\kappa_1 \geq 2$ and $\kappa_2  < \infty$.
\end{proposition}

\begin{proof}
Note that \eqref{eq:exp} yields \eqref{eq:fsm}, while \eqref{eq:comp_dens}, \eqref{eq:suf1} are consequences of Proposition
\ref{prop:bound_by_psi}. The first bound on $K^X_1$ follows by the estimates in \cite[Prop. 2, p. 22]{bib:KS14}. To
prove the second bound, let $x_{s_2}=(s_2,0,...,0)$, $y_{s_1}=(s_1,0,...,0)$ and observe that
$$
g(|x_{s_2}-y_{s_1}|) = g(|x_{s_2}|) \frac{g(|x_{s_2}-y_{s_1}|)}{g(|x_{s_2}|)} \geq  e^{c s_1} \left(\frac{s_2}{s_2-s_1}
\right)^{\delta} g(|x_{s_2}|) \geq e^{c s_1} g(|x_{s_2}|),
$$
which gives
$$
K^X_2(s_1,s_2,\infty) \geq \frac{\nu(x_{s_2}-y_{s_1})}{\nu(x_{s_2})}\geq c_1 e^{c s_1}, \quad s_2  > 2s_1 \geq 2.
$$
In particular, it is directly seen that none of the conditions \eqref{eq:intr_killing} and \eqref{eq:unif_bdd} can hold
for any $\kappa_1 \geq 2$ and $\kappa_2  < \infty$.
\end{proof}

As seen from Theorem \ref{thm:defic6} and Corollary \ref{prop:exp}, a sufficiently low-lying ground state eigenvalue
$\lambda_0<0$ secures good decay properties of the corresponding eigenfunction. An important application of this
type of spectral information is in mathematical physics in studies of
the stability of quantum systems with nuclear repulsion. For the case of one-electron quasi-relativistic atoms with Coulomb interaction see \cite{bib:DL}.

By using the following simple and intuitively clear criterion, we can easily settle when is $\lambda_0<0$ small enough
in comparison with the magnitude of the potential. Denote by $\mu_r  >0$ the principal eigenvalue of the semigroup
of the process $(X_t)_{t \geq 0}$ killed on exiting the ball $B(0,r)$, $r>0$.

\begin{proposition}
\label{smallev}
Let $\pro X$ be a L\'evy process with L\'evy-Khinchin exponent $\psi$ given by \eqref{eq:Lchexp}-\eqref{eq:nuinf}
and let Assumption (A4) hold. Suppose that a ground state at the eigenvalue $\lambda_0 < 0$ exists. Then for every $r >0$ we have
$$
\lambda_0  \leq \sup_{|y| \leq 2r} V_+(y) - \inf_{|y| \leq 2r} V_-(y) + \mu_r.
$$
\end{proposition}

\begin{proof}
For every $r>0$ denote $c_1=c_1(r):=1/|B(0,2r)|$. We have for all $t>0$ that
\begin{align*}
c_1 \int_{B(0,r)} \left(\ex^x\left[\tau_{B(x,r)} >t; e^{-\int_0^t V(X_s)ds} \right] \right)^2 dx & \leq
\int_{\R^d} \left(\ex^x\left[e^{-\int_0^t V(X_s)ds} \, \sqrt{c_1} \1_{B(0,2r)}(X_t) \right] \right)^2 dx \\
& \leq \left\|T_t\right\|^2_{2,2} = e^{-2\lambda_0 t}.
\end{align*}
For the paths starting at $x \in B(0,r)$ and satisfying $\tau_{B(x,r)} >t$ we have
$$
\int_0^t V(X_s)ds \leq t\left(\sup_{|y| \leq 2r} V_+(y) - \inf_{|y| \leq 2r} V_-(y)\right).
$$
Hence we get
$$
c_1 |B(0,r)| (\pr^0(\tau_{B(0,r)} >t))^2 e^{-2t(\sup_{|y| \leq 2r} V_+(y) - \inf_{|y| \leq 2r} V_-(y))} \leq
e^{-2\lambda_0 t}, \quad t>0, \; r>0.
$$
Since $\pr^0(\tau_{B(0,r)} >t) \geq c_2 e^{-\lambda_r t}$ for all $t>0$ with the same constant $c_2=c_2(r)$, the
above estimate gives
$$
2\lambda_0 t \leq 2t\left(\sup_{|y| \leq 2r} V_+(y) - \inf_{|y| \leq 2r} V_-(y)\right) + 2 \lambda_r t -
\log \left[c_1 c_2^2 |B(0,r)|\right], \quad t>0, \; r>0.
$$
Dividing on both sides by $2t$ and taking the limit $t \to \infty$ the claim follows.
\end{proof}

\begin{corollary}  [\textbf{Super-exponentially decaying L\'evy intensities}]
\label{prop:supexp}
\noindent
Let $\pro X$ be a L\'evy process with L\'evy-Khinchin exponent $\psi$ given by \eqref{eq:psi_ex}--\eqref{eq:nu_g} with $a=0$.
Moreover, let \eqref{eq:weakscaling} with $\gamma_1=\gamma_2$ be satisfied and assume
\begin{align} \label{eq:supexp}
g(r) = e^{-c r^{\beta}} r^{-\delta}, \quad r \geq 1, \quad \text{with} \ \ c>0, \ \ \beta > 1 \ \  \text{and} \ \ \delta \geq 0.
\end{align}
If assumption (A4) holds and a ground state at eigenvalue $\lambda_0 < 0$ exists, then there is a constant
$C_{37}    > 0$ such that for every $\varepsilon >0$ there exist $C_{38} >0$ and $R>0$ for which
$$
\varphi_0(x) \geq C_{38} e^{-C_{37} \sqrt{(|\lambda_0|+\varepsilon) \wedge 1} \, |x| (\log|x|)^{(\beta-1)/\beta}} ,
\quad |x| \geq R.
$$
\end{corollary}

\begin{proof}
By the lower estimate in (1.17) of \cite[Th. 1.2]{bib:CKK}, for sufficiently large $|x|$
$$
G^{\eta}(x) \geq \int_1^{|x|} e^{-\eta t} p(t,x) dt \geq c_1  e^{-c_2\left(|x| (\log|x|)^{\frac{\beta-1}{\beta}}\right)}
\int_1^{|x|} e^{-\eta t} t^{-d/2} dt, \quad \eta>0,
$$
and
$$
G^{\eta}(x) \geq \int_{\frac{|x|}{\sqrt{\eta} (\log|x|)^{\beta/(\beta-1)}}}^{\infty} e^{-\eta t} p(t,x) dt
\geq c_1 e^{-c_2 \sqrt{\eta} \, \left(|x| (\log|x|)^{\frac{\beta-1}{\beta}}\right)}
\int_{\frac{|x|}{\sqrt{\eta} (\log|x|)^{\beta/(\beta-1)}}}^{\infty} e^{-\eta t} t^{-d/2} dt,
$$
whenever $\eta \in (0,1]$. With this, the result follows by a similar argument as in Corollary \ref{prop:exp}(3).
\end{proof}

Before closing this subsection we make a point about the mechanism behind the phase transition.
\begin{remark} [\textbf{Phase transition in the decay rates}]
\label{rem:exp_trans}
{\rm
Denote by $\nu_1(dx) = \nu(x) \1_{\left\{|x| \geq 1\right\}} dx$ the L\'evy measure of the compound Poisson process
part, which we denote by $\pro {\overline {X}}$, in the L\'evy-It\^o decomposition of $\pro X$. Recall that the
subsequent jumps $J_1, J_2, J_3,...$ of $\pro {\overline{X}}$ are independent, identically distributed random variables
on a  common probability space $(\overline \Omega, \overline \pr)$, distributed by the probability measure $\overline
\nu_1(dx) = \nu_1(dx)/\nu_1(\R^d)$. Observe that for a symmetric jump-paring process $\pro X$ we have by (A1.3) that
\begin{align} \label{eq:jump_distr}
\overline \pr(|J_1 + J_2| > r) \leq c_1 \overline \pr(|J_1| > r),
\quad \text{for large \ $r>0$,}
\end{align}
with some $c_1 > 0$. Obviously, the smallness conditions \eqref{eq:intr_killing}-\eqref{eq:unif_bdd} are actually
governed by the jump properties of $\pro {\overline {X}}$. On the level of the specific L\'evy measures with ``nice"
enough profiles discussed in detail in the previous two subsections we see that the transition in the ground state
fall-off rates occurs exactly when the decay rate of $\nu$ changes in leading order from strictly sub-exponential to
exponential. In particular, the fall-off is driven by $\nu$ as long as
\begin{align} \label{eq:linear}
\lim_{|x| \to \infty} \frac{\log \nu(x)}{|x|} = \lim_{|x| \to \infty} \frac{\log \overline \nu_1(x)}{|x|} = 0.
\end{align}
One can conjecture that this qualitative change in the decay rates is strongly related to a change in the mechanism
which makes the long jumps of the process occur. For simplicity, consider just the $d=1$ case. We can make the
following comparison with the concept of subexponential distributions in probability \cite{bib:EGV}. It can be
checked that for all the specific profiles discussed in Subsection \ref{subsec:conseq} the measures $\overline
\nu_1(dx)$ are cases of subexponential distributions, i.e., $\overline \pr(|J_1 + J_2| > r) \approx 2 \overline
\pr(|J_1| > r)$ as $r \to \infty$. Taken along with the general property $\overline \pr(|J_1| \vee |J_2| > r) \approx 2
\overline \pr(|J_1| > r)$, this implies $\overline \pr(|J_1 + J_2| > r) \approx \overline \pr(|J_1| \vee |J_2| > r)$ as
$r \to \infty$, which means that a double jump $J_1+J_2$ is larger than a large $r$ when either $J_1$ or $J_2$ is
larger than $r$. It is much less likely that both $J_1, J_2$ are less than $r$, but they are large enough
so that their sum exceeds $r$. This property extends to arbitrarily long sequences of jumps. When $\pro X$ is a
jump-paring process but $\overline \nu_1(dx)$ is not subexponential, this is no longer true. From \eqref{eq:jump_distr}
we see that the ratio $\overline \pr(|J_1 + J_2| > r)/\overline \pr(|J_1| > r)$ is bounded above for large $r$, but we only
have $\liminf_{r \to \infty} \overline \pr(|J_1 + J_2| > r)/\overline \pr(|J_1| > r) > 2$. This gives
$\liminf_{r \to \infty} \overline \pr(|J_1 + J_2| > r)/ \overline \pr(|J_1| \vee |J_2| > r) > 1$ which means that the
above described mechanism weakens. This can be seen as an increase of the capacity of the process to fluctuate through
multiple smaller jumps rather than exceedingly large single jumps, which improves as the tail of $\overline \nu_1(dx)$,
and thus of $\nu$, gets lighter.
}
\end{remark}

\subsection{Upper bound for confining potentials}

Finally we show that by using Lemma \ref{lem:defic1} one of our recent results on eigenfunction estimates in \cite{bib:KL14}
can be recovered. In this brief subsection we consider \emph{confining potentials} in the sense of the following condition:

\medskip
\begin{itemize}
\item[\textbf{(A5)}]
Let $V \in \cK_{\pm}^X$ be such that $V(x) \to \infty$ as $|x| \to \infty$.
\end{itemize}

\medskip
\noindent
Note that under Assumption (A5) the operators $T_t:L^2(\R^d) \to L^2(\R^d)$, $t  >0$, are compact, which follows by standard
arguments based on approximation of $T_{t}$, $t \geq 2t_{\rm b}$, in terms of compact operators \cite[Lem. 3.2]{bib:KL}. Clearly,
compactness extends to all $t>0$ by the spectral theorem. The theory of operator semigroups implies then that there
exists an orthonormal basis in $L^2(\R^d)$ consisting of the eigenfunctions $\varphi_n$ given by $T_t \varphi_n = e^{-\lambda_n t}
\varphi_n$, $t>0$, $n \geq 0$, and the spectrum of $H$ consists only of eigenvalues $\lambda_0 < \lambda_1 \leq \lambda_2 \leq ... \to
\infty$ of finite multiplicity each. All $\varphi_n$ are bounded continuous functions, and the ground state $\varphi_0$ has a
strictly positive version.

The following result has been obtained first in \cite[Th. 2.1]{bib:KL14}. For the sake of completeness, we show that in our present
framework it can be obtained without jump estimates.

\begin{theorem}
\label{thm:defic_conf}
Let $\pro X$ be a L\'evy process with L\'evy-Khinchin exponent $\psi$ given by \eqref{eq:Lchexp}-\eqref{eq:nuinf}
such that Assumptions (A1)-(A3) and (A5) hold. Suppose that $\varphi$ is an eigenfunction corresponding to the
eigenvalue $\lambda \in \R$. Then there exist $C_{39}=C_{39}(X,V) >0$ and $R=R(X,V) > 0$ such that
$$
|\varphi(x)| \leq C_{39} \left\|\varphi\right\|_{\infty} \, \nu(x), \quad |x| > R.
$$
\end{theorem}

\begin{proof}
Choose $\theta \in \R$ such that $\theta+\lambda>0$ and let $r > 0$ be large enough so that $\kappa:=\inf_{|y| \geq r} V(y) >
(\eta_0 + \lambda) \vee 0$, where $\eta_0$ is given by \eqref{eq:eta0}. Denote $\tau_r := \tau_{\overline B(0,r)^c}$. Similarly as in
Theorem \ref{thm:defic6}, an application of \eqref{eq:eig1} to $D=\overline B(0,r)^c$ gives for every $|x| \geq r+1$ that
\begin{align*}
|\varphi(x)| & \leq (\theta+\lambda) \ex^x\left[ \int_0^{\tau_{r}} e^{-(\theta+\kappa)t} |\varphi(X_t)| dt\right] +
\ex^x\left[\tau_{r} < \infty; e^{-\int_0^{\tau_{r}}(\theta+V(X_s))ds} |\varphi(X_{\tau_{r}})| \right] \\
& \leq  \left\|\varphi\right\|_{\infty} \left(\frac{\theta+\lambda}{\theta+\kappa} + \ex^x\left[e^{-(\theta+\kappa) \tau_{r}} \right]
\right).
\end{align*}
We complete the proof by choosing $\theta = - \lambda + \delta$ with $\delta   >0$. Then by taking the limit $\delta \downarrow 0$ and by application of Lemma \ref{lem:defic1} with
$\eta=\kappa-\lambda > \eta_0$ the claim follows.
\end{proof}
\noindent
Note that the above result is a key step in deriving sharp bounds for eigenfunctions for confining potentials \cite[Th. 2.3-2.4]{bib:KL14}.

\medskip

\textbf{Acknowledgements.} We thank the Referee for a careful reading of the manuscript.

\end{document}